\documentclass[a4paper,11pt]{article}
\usepackage{cmap}
\usepackage[T2A]{fontenc}
\usepackage[cp1251]{inputenc}
\usepackage[english,russian]{babel}
\usepackage{amsfonts,amssymb,amsthm,amsmath}

\usepackage{amsbib}

\usepackage{url}
\usepackage{enumitem}

\usepackage{secdot}

\usepackage{graphicx}
\usepackage{epstopdf}
\usepackage{varwidth}

\usepackage{a4wide}
\usepackage{xcolor}

\allowdisplaybreaks
\usepackage{mathtools}
\mathtoolsset{showonlyrefs}

\usepackage{titlesec}
\titleformat{\subsection}[runin]
{\normalfont\large\bfseries}{\thesubsection}{1em}{}

\let\OLDthebibliography\thebibliography
\renewcommand\thebibliography[1]{
	\OLDthebibliography{#1}
	\setlength{\parskip}{1pt}
	\setlength{\itemsep}{1pt plus 0.3ex}
}

\numberwithin{equation}{section}
\theoremstyle{plain}
\newtheorem{theorem}{Theorem}
\newtheorem{lemma}{Lemma}[section]
\newtheorem{propos}{Proposition}
\theoremstyle{definition}
\newtheorem{definition}{Definition}
\newtheorem{remark}{Remark}

\def\P{\mathrm{P}}
\def\L{\mathrm{L}}
\def\x{\mathrm{x}}

\def\a{\mathrm{a}}
\def\A{\mathrm{A}}

\def\p{\mathrm{p}}

\def\s{\mathrm{s}}
\def\t{\mathrm{t}}
\def\v{\mathrm{v}}
\def\c{\mathrm{c}}
\def\f{\mathrm{f}}

\def\RR{\mathbb R}

\title{
	\vspace*{-0.5cm}
	On solutions of anisotropic elliptic equations with variable exponent  and measure data
}

\makeatletter
\newcommand{\address}[1]{\gdef\@address{#1}}
\newcommand{\email}[1]{\gdef\@email{\url{#1}}}
\newcommand{\homepage}[1]{\gdef\@homepage{\url{#1}}}
\newcommand{\@endstuff}{\vspace{\baselineskip}\noindent\small
	\begin{tabular}{@{}l}
		\@address\\
		\textit{E-mail address:} \@email\\
\end{tabular}}
\AtEndDocument{\@endstuff}
\makeatother

\author{L.\,M.~Kozhevnikova}
\email{kosul@mail.ru}
\address{
	\textsc{Sterlitamak branch of Bashkir State University},\\
	\textsc{Elabuga Institute of Kazan Federal University}}

\date{}

\begin{document}
	\maketitle
	\selectlanguage{english}
	\begin{abstract}
		The Dirichlet problem in arbitrary domains for a wide class of anisotropic elliptic equations of the second order with variable exponent nonlinearities and the right-hand side as a measure is considered.
		The existence of an entropy solution in anisotropic Sobolev spaces with variable exponents is established.
		It is proved that the obtained entropy solution is a renormalized solution of the considered problem.
		
		
	\end{abstract}

\section{Introduction}\label{s1}

Since the end of the last century, nonlinear elliptic equations of the second order
\begin{equation}\label{urf}
-{\rm div}\,(\a(\x,u,\nabla
u)+\c(u))+a_0(\x,u,\nabla
u)=\mu,\quad \x\in \Omega,
\end{equation}
with a measure on the right-hand side have been intensively studying.
Hereinafter, $\Omega$ is a domain in $\mathbb{R}^n=\{\x=(x_1,x_2,\ldots,x_n)\},\;$ $\Omega\subsetneq \mathbb{R}^n,\;n \geq 2$, $${\a}(\x,s_0,\s)=(a_1(\x,s_0,\s),\ldots,a_n(\x,s_0,\s)):\Omega\times\mathbb{R}\times\mathbb{R}^n\rightarrow \mathbb{R}^n,$$
 $$a_0(\x,s_0,\s):\Omega\times\mathbb{R}\times\mathbb{R}^n\rightarrow \mathbb{R},\quad{\c}(s_0)=(c_1(s_0),\ldots,c_n(s_0)):\mathbb{R}\rightarrow \mathbb{R}^n.$$

Ph.\ Benilan,  L.\ Boccardo,  Th.\  Gallou\"{e}t, R.\ Gariepy,  M.\ Pierre, J.~L.\ Vazquez in \cite{Beni} proposed a notion of entropy solution of the Dirichlet problem for the elliptic equations with power nonlinearities
 \begin{equation}\label{ur0}
-{\rm div}\,\a(\x,\nabla
u)+a_0(\x,u)=\mu
\end{equation}
where $\mu\in L_1(\Omega)$, and they proved the existence and uniqueness of such solution.
Instead of the entropy solution firstly introduced by S.~N.\ Kruzhkov \cite{Kru} for the first order equations, it is also possible to consider a renormalized solution. A notion of renormalized solution was firstly introduced in  \cite{LionDi} to study the Cauchy problem for the Boltzmann equation.
Such solutions are elements of the same functional class as entropy solutions, but satisfy a different family of integral relations.
In some cases, notions of entropy and renormalized solutions are equivalent.

L.\ Boccardo in \cite{Boccardo_1996} proved the existence of solutions of the Dirichlet problem in bounded domains $\Omega$ for the equation \eqref{urf} with $a_0(\x,s_0,\s)\equiv 0,\;\c(s_0)\in C_0(\mathbb{R},\mathbb{R}^n)$ and $\mu\in L_1(\Omega)$.
In \cite{Boccardo_1993}, it was proved the existence of renormalized solutions and some regularity results to the Dirichlet problem in bounded domains $\Omega$ for equations of the type \eqref{urf} with the function $a_0(\x,s_0,\s
)=a_0(\x,s_0)$ having some growth and sign conditions, and $\mu\in W^{-1}_{p'}(\Omega)$.
The existence of entropy solutions of the Dirichlet problem in bounded domains for the equation \eqref{urf} with $\c(s_0)\equiv 0$, $a_0(\x,s_0)\equiv 0$, $\mu\in L_1(\Omega)$ and a degenerate coercivity was established by A.~A.\ Kovalevsky \cite{Koval_UMJ}.

Questions about the existence and uniqueness of renormalized and entropy solutions of the Dirichlet problem for elliptic equations of the second order with non-power nonlinearities and $\mu\in L_{1}(\Omega)$  ($\Omega$ is a bounded domain) in Sobolev-Orlicz spaces were studied in \cite{BB},  \cite{ABT}, \cite{GWWZ}.
Theorems on the existence and uniqueness of entropy solutions of the Dirichlet problem in arbitrary domains for a class of anisotropic elliptic equations with non-power nonlinearities in Sobolev-Orlicz spaces were proved by the author in \cite{Kozh_vmmf_2017}, \cite{Kozh_int_2017}.
Since then, a lot of articles have been devoted to these issues, see the surveys \cite{Muk}, \cite{Muk_2018}.

At present, the study of quasilinear equations with a measure data has become a mature subject of research.
The first investigations were concerned with the Dirichlet problem for the equation
$$
-\Delta_pu=\mu
$$
in a smooth bounded domain $\Omega\subset \mathbb{R}^n$, where $\mu$ is a Radon measure on $\Omega$.
Initially, solutions were understood in the distributional sense, but neither the uniqueness nor any kind of stability have been studied. Important progress was achieved with the introduction of renormalized solutions, and in the framework of this notion a strong convergence of the gradient and a stability of solutions were established.

The concept of renormalized solutions is the main step in the study of general degenerate elliptic equations whose data is a measure.
The initial definition was given in \cite{DMOP_1999} and then extended by M.~F.\ Bidaut-Veron in \cite{Bidaut_Veron} to the local form. The main conclusion in \cite{DMOP_1999} is the fact that each Radon measure $\mu$ of a bounded variation can be decomposed as $\mu = \mu_0 + \mu_s$, where $\mu_0 \in L_1 (\Omega) + W^{-1}_{p'} (\Omega)$ and $\mu_s$ is concentrated on a set of the zero $p$-capacity. In \cite{DMOP_1999} and \cite{Malusa}, the existence and stability of a renormalized solution of the Dirichlet problem for the equation \eqref{ur0} with $a_0(\x,s_0)\equiv 0$ were proved.
In \cite{Malusa_Porzio}, for the equation \eqref{ur0} there was proved the existence, and for $\mu =
\mu_0$ also the uniqueness, of a renormalized solution of the Dirichlet problem in an arbitrary domain $\Omega$.
A detailed survey of results for quasilinear degenerate equations with power nonlinearities and a measure data can be found in the monograph of L.~Veron \cite{Veron}.

On the other hand, since the end of the last century, differential equations and variational problems associated with assumptions of $p(\x)$-growth have been also widely studying.
Interest in these investigations is due to the fact that such equations can be used to model various phenomena arising in the study of electrorheological and thermoreological liquids, elasticity, and image reconstruction \cite{Halsey}.

The modern theory of elliptic equations with non-standard growth conditions was developed by V.~V.\ Zhikov \cite{Zhik}, Yu.~A.\ Alkhutov \cite{Alxut}. In the works \cite{BendWitt}, \cite{Urbano}, \cite{Ouaro}, \cite{Zhang_Zhou}, there were proved theorems on the existence and uniqueness of renormalized and entropy solutions to the Dirichlet problem for equations with variable exponent nonlinearities in bounded domains $\Omega$.
In \cite{Kozhe_SMFN}, for the anisotropic equation \eqref{ur0} with variable exponent nonlinearities with $a_0(\x,s_0)=|s_0|^{p_0(\x)-2}s_0+b(\x,s_0)$, nondecreasing function $b(\x,s_0)$ with respect to $s_0$, and $\mu\in L_1(\Omega)$, the existence and uniqueness of a entropy solution  for the Dirichlet problem in an arbitrary unbounded domain $\Omega$ were proved.

Let us denote
$$
C^+(\overline{\Omega})=\{p\in C(\overline{\Omega})\;:\; 1<p^-\leq p^+<+\infty \},
$$
where $p^-=\inf\limits_{\x\in \Omega}p(\x),\;$ $p^+=\sup\limits_{\x\in \Omega}p(\x)$.
Let
$p(\cdot)\in C^+(\overline{\Omega})$. We define the Lebesgue space with variable exponent
$L_{p(\cdot)}(\Omega)$ as the set of measurable on $\Omega$ real-valued functions $v$ such that
$$
\rho_{p(\cdot),\Omega}(v)=\int\limits_\Omega|v(\x)|^{p(\x)}d\x<\infty.
$$
The Luxemburg norm in $L_{p(\cdot)}(\Omega)$ is defined by
$$
\|v\|_{L_{p(\cdot)}(\Omega)}=\|v\|_{p(\cdot),\Omega}=\inf\left\{k>0\;\Big |\;
\rho_{p(\cdot),\Omega}(v/k)\leq1\right\}.
$$
The Sobolev space with variable exponent $\mathring
{H}_{p(\cdot)}^1(\Omega)$ is defined as a completion of $C_0^{\infty}(\Omega)$ with respect to the norm
$$
\|v\|_{\mathring
{H}_{p(\cdot)}^1(\Omega)}=\|\nabla v\|_{p(\cdot),\Omega}.$$

The set of bounded Radon measures is denoted as $\mathcal{M}^b(\Omega)$. The measure $\mu\in\mathcal{M}^b(\Omega)$ is called diffuse if $\mu(E)=0$ for any $E$ such that ${\rm Cap}_{p(\cdot)}\,(E,\Omega)=0$. Here, $p(\cdot)$-capacity of a subset $E$ with respect to $\Omega$ is defined as
$$
{\rm Cap}_{p(\cdot)}\,(E,\Omega)=\inf\limits_{S_{p(\cdot)}(E)}\rho_{p(\cdot),E}(|\nabla v|),
$$
$$S_{p(\cdot)}(E)=\left\{v\in\mathring
{H}_{p(\cdot)}^{1}(\Omega)\cap C_0(\Omega)\;\Big |\; v(\x)=1,\;\x\in E,\; v(\x)\geq 0,\;\x\in \Omega  \right\}.$$
We denote by $\mathcal{M}^b_{p(\cdot)}(\Omega)$ the set of all bounded Radon
diffuse measures.
In the case of a bounded domain $\Omega$, it was proved in \cite{Ouaro_2013}, \cite{Zhang_2014} that $\mu\in\mathcal{M}^b_{p(\cdot)}(\Omega)$ if and only if $\mu\in L_1(\Omega)+H^{-1}_{p'(\cdot)}(\Omega)$.
For the anisotropic case such representation is not known.

The existence of entropy solutions of the Dirichlet problem in bounded domains $\Omega$ for equations with variable exponent nonlinearities of the type \eqref{urf} was studied in \cite{AHT}, \cite{Benboubker_2014}, \cite{Benboubker_2015}, \cite{Azroul_2018}. Namely, it was proved in \cite{AHT}, \cite{Benboubker_2014} that for $\mu\in L_1(\Omega)$ there exists an entropy solution of the equation \eqref{urf} under homogeneous boundary conditions
\begin{equation}
\label{gu}
\left.u\right|_{\partial \Omega}=0.
\end{equation}
The authors of \cite{Benboubker_2015}, \cite{Azroul_2018} established the existence of an entropy solution of the problem \eqref{urf}, \eqref{gu} for $\mu\in L_1(\Omega)+H^{-1}_{p'(\cdot)}(\Omega)$, that is,
$$
\mu=f-{\rm div}\, \f,\quad f\in L_1(\Omega),\quad \f=(f_1,\dots,f_n)\in (L_{p'(\cdot)}(\Omega))^n.
$$

In the works \cite{Benboubker_2015}, \cite{Azroul_2018}, it is assumed that $\c\in C_0(\mathbb{R},\mathbb{R}^n),\;$ $\a(\x,s_0,\s)$ is a Caratheodory function and there exist a nonnegative function $\Phi\in L_{p'(\cdot)}(\Omega)$ and  positive numbers $\widehat{a}, \overline{a}$ such that for a.a.\ $\x\in\Omega$ and for any $s_0\in \mathbb{R},\;$  $\s,\t\in\mathbb{R}^n,\;$ the following inequalities are satisfied:
\begin{equation}
\label{ogr}
|\a(\x,s_0,\s)|\leq \widehat{a}\left(|s_0|^{p(\x)-1}+|\s|^{p(\x)-1}
+\Phi(\x)\right);
\end{equation}
\begin{equation}
\label{mon}
(\a(\x,s_0,\s)-\a(\x,s_0,\t))\cdot(\s-\t)>0,\quad \s\neq \t;
\end{equation}
\begin{equation}
\label{koerc}
\a(\x,s_0,\s)\cdot\s\geq \overline{a}|\s|^{p(\x)}.
\end{equation}
Here ${\s}\cdot{\t}=\sum\limits_{i=1}^ns_it_i,\;$ $\s=(s_1,\ldots,s_n),\;$$\t=(t_1,\ldots,t_n)$.

It should be pointed out that denoting $\widetilde{\a}(\x,s_0,\s)=\a(\x,s_0,\s)-\f$ we get the equation
$$
-{\rm div}(\widetilde{\a}(\x,u,\nabla
u)+\c(u))+a_0(\x,u,\nabla u)=f
$$
with the function $\widetilde{\a}(\x,s_0,\s)$ which satisfies conditions of the form \eqref{ogr}, \eqref{mon}.
Moreover, the coercivity assumption \eqref{koerc} becomes
$$
\widetilde{\a}(\x,s_0,\s)\cdot\s\geq \widetilde{a}|\s|^{p(\x)}-\phi(\x),\quad \phi\in L_1(\Omega).
$$

In the present paper, this idea is realized for the anisotropic equation of the form
\begin{equation}
\label{ur}
{\rm div}\,\a(\x,u,\nabla
u)=|u|^{p_0(\x)-2}u+b(\x,u,\nabla u)+\mu,\quad \x\in \Omega
\end{equation}
in arbitrary domains $\Omega\subsetneq \mathbb{R}^n$ and with functions $a_i(\x,s_0,\s)$ satisfying less restrictive requirements (see the assumptions \eqref{us1}--\eqref{us3}) than in the works \cite{Benboubker_2015}, \cite{Azroul_2018}.

Therefore, to the best of our knowledge, in the existing literature there are results for entropy and renormalized solutions of elliptic problems in bounded domains (except for the works \cite{Beni}, \cite{Bidaut_Veron}, \cite{Malusa_Porzio} for equations with power nonlinearities, and the works of the author).
In the present paper, it is proved the existence of entropy solutions of the Dirichlet problem \eqref{ur},
\eqref{gu} in anisotropic Sobolev spaces with variable exponents without assuming the boundedness of the domain $\Omega$.
Moreover, it is established that the obtained solution is a renormalized solution of the Dirichlet problem \eqref{ur}, \eqref{gu}.

\section{Anisotropic Sobolev space with variable exponents}\label{s2}

Let $Q\subsetneq\mathbb{R}^n$ be an arbitrary domain and $p(\cdot)\in C^+(\overline{Q})$.
The following Young's inequality is satisfied:
\begin{equation}\label{ung}
|y z|\leq |y|^{p(\x)}+|z|^{p'(\x)},\quad z,y\in \mathbb{R},\quad \x\in Q, \quad p'(\x)=\frac{p(\x)}{p(\x)-1}.
\end{equation}
Moreover, in view of the convexity, there holds
\begin{equation}\label{sum}
|y+z|^{p(\x)}\leq 2^{p^+-1}(|y|^{p(\x)}+|z|^{p(\x)}),\quad z,y\in \mathbb{R},\quad\x\in Q.
\end{equation}

Hereinafter, we will use the notations $\|v\|_{p(\cdot),\Omega}=\|v\|_{p(\cdot)},\;\rho_{p(\cdot),\Omega}(v)=\rho_{p(\cdot)}(v).$
The norm in $L_p(Q),\; p\in[1,\infty],$ will be denoted by $\|v\|_{p,Q}$, and $\|v\|_{p,\Omega}=\|v\|_{p}$.
The space $L_{p(\cdot)}(Q)$ is a separable reflexive Banach space \cite{Diening}.

For any $u\in L_{p'(\cdot)}(Q),\;v\in L_{p(\cdot)}(Q)$, the following H\"older inequality is satisfied:
\begin{equation}\label{gel}
  \int\limits_Q |u(\x)v(\x)|d\x\leq 2\|u\|_{p'(\cdot),Q}\|v\|_{p(\cdot),Q}.
\end{equation}
Moreover, the following relations hold true \cite{Diening}:
\begin{gather}
\|v\|^{p^-}_{p(\cdot),Q}-1\leq\min\{\|v\|^{p^-}_{p(\cdot),Q},\|v\|^{p^+}_{p(\cdot),Q}\}\leq\rho_{p(\cdot),Q}(v)\leq\nonumber\\
\leq\max\{\|v\|^{p^-}_{p(\cdot),Q},\|v\|^{p^+}_{p(\cdot),Q}\}\leq \|v\|^{p^+}_{p(\cdot),Q}+1,\label{st2}
\end{gather}
and
\begin{gather}
\left(\rho_{p(\cdot),Q}(v)-1\right)^{1/p^+}\leq\min\{\rho^{1/p^+}_{p(\cdot),Q}(v),\rho^{1/p^-}_{p(\cdot),Q}(v)\}\leq\nonumber\\
\leq\|v\|_{p(\cdot),Q}\leq \max\{\rho^{1/p^+}_{p(\cdot),Q}(v),\rho^{1/p^-}_{p(\cdot),Q}(v)\}\leq\left(\rho_{p(\cdot),Q}(v)+1\right)^{1/p^-}.\label{st3}
\end{gather}
Furthermore, for $p(\cdot)\in L_{\infty}(Q)$ such that
$1\leq p(\x)q(\x)\leq \infty$ for a.a.\ $\x\in Q$ and $v\in L_{q(\cdot)}(Q),\;v\not\equiv 0$ we have (see \cite{Edmunds})
\begin{equation}\label{st4}
\min\{\|v\|^{p^+}_{p(\cdot)q(\cdot),Q},\|v\|^{p^-}_{p(\cdot)q(\cdot),Q}\}\leq\||v|^{p(\cdot)}\|_{q(\cdot),Q}\leq \max\{\|v\|^{p^+}_{p(\cdot)q(\cdot),Q},\|v\|^{p^-}_{p(\cdot)q(\cdot),Q}\}.
\end{equation}

Let us denote
$$
\overrightarrow{\p}(\cdot)=(p_1(\cdot),p_2(\cdot),\ldots,p_n(\cdot))\in (C^+(\overline{Q}))^n,
\quad
\overrightarrow{\bf p}(\cdot)=(p_0(\cdot),\overrightarrow{\p}(\cdot))\in (C^+(\overline{Q}))^{n+1},
$$
and define
$$
p_+(\x)=\max_{i=\overline{1,n}}p_i(\x),\quad p_-(\x)=\min_{i=\overline{1,n}}p_i(\x),\quad  \x\in Q.
$$
We will also use the notations $\frac{\partial v}{\partial x_i}=v_{x_i}=\partial_iv,\;i=1,\ldots,n$.
Anisotropic Sobolev spaces with variable exponents $\mathring
{H}_{\overrightarrow{\p}(\cdot)}^1(Q),\; \mathring
{W}_{\overrightarrow{\bf p}(\cdot)}^{1}(Q)$ are defined as completions of the space $C_0^{\infty}(Q)$ with respect to the norms
$$
\|v\|_{\mathring
{H}_{\overrightarrow{\p}(\cdot)}^1(Q)}=\sum_{i=1}^n\|v_{x_i}\|_{p_i(\cdot),Q},$$
$$
\|v\|_{\mathring
{W}_{\overrightarrow{\bf p}(\cdot)}^{1}(Q)}=\|v\|_{p_0(\cdot),Q}+\|v\|_{\mathring
{H}_{\overrightarrow{\p}(\cdot)}^{1}(Q)},
$$
respectively.
The spaces $\mathring{H}_{\overrightarrow{\p}(\cdot)}^1(Q),\;\mathring{W}_{\overrightarrow{\bf p}(\cdot)}^{1}(Q)$ are reflexive Banach spaces \cite{Fan}.

Let
$$
\overline{p}(\x)={n}\left(\sum\limits_{i=1}^n 1/p_i(\x)\right)^{-1},\quad
p_*(\x)=\left\{\begin{array}{ll}\frac{n\overline{p}(\x)}{n-\overline{p}(\x)},& \overline{p}(\x)<n,\\
+\infty,& \overline{p}(\x)\geq n,
\end{array}\right.$$$$
p_{\infty}(\x)=\max\{p_*(\x),p_+(\x)\},\quad \x\in Q.
$$
We recall the following embedding theorem for the space $\mathring{H}_{\overrightarrow{\p}(\cdot)}^1(Q)$, see \cite[Theorem 2.5]{Fan}.
\begin{lemma}\label{lemma_0}
	Let $Q$ be a bounded domain and $\overrightarrow{\p}(\cdot)=(p_1(\cdot),\ldots,p_n(\cdot))\in (C^+(\overline{Q}))^n$. If $q(\cdot)\in C^+(\overline{Q})$ and
\begin{equation}\label{kozhelm-1}
q(\x)<p_{\infty}(\x)\quad \forall\;\x\in Q,
\end{equation}
then the embedding $\mathring
{H}_{\overrightarrow{\p}(\cdot)}^1(Q)\hookrightarrow L_{q(\cdot)}(Q)$ is continuous and compact.
\end{lemma}

\section{Assumptions and main results}\label{s3}

Let $\overrightarrow{\bf p}(\cdot)=(p_0(\cdot),p_1(\cdot),\ldots,p_n(\cdot))\in (C^+(\overline{\Omega}))^{n+1}$.
We will assume that
\begin{equation}\label{us0}
p_+(\x)\leq p_0(\x),\quad \x\in \Omega.
\end{equation}
We will also assume that functions
$a_i(\x,s_0,\mathrm{s}),$ $i=1,\ldots,n,$ and $b(\x,s_0,\mathrm{s})$ from the equation \eqref{ur} are measurable with respect to $\x\in \Omega$ for $s_0\in\mathbb{R},\; \mathrm{s}=(s_1,\ldots,s_{n})\in\mathbb{R}^{n}$, and continuous with respect to $(s_0,\mathrm{s})\in\mathbb{R}^{n+1}$ for a.a.\ $\x\in\Omega$.
Assume that there are nonnegative functions $\Phi_i\in L_{p'_i(\cdot)}(\Omega),\;\phi\in L_1(\Omega)$, continuous nondecreasing functions $\widehat{a}_i:\mathbb{R}^+\rightarrow \mathbb{R}^+\setminus\{0\},\;i=1,\ldots,n,$ and a positive number $\overline{a}$, such that for a.a.\ $\x\in\Omega$ and all $s_0\in \mathbb{R},\;$  $\s,\t\in\mathbb{R}^n$  the following inequalities are satisfied:
\begin{equation}\label{us1}
|a_i(\x,s_0,\s)|\leq \widehat{a}_i(|s_0|)\left((\P(\x,\s))^{1/p'_i(\x)}
+\Phi_i(\x)\right),\quad i=1,\ldots,n;
\end{equation}
\begin{equation}
(\a(\x,s_0,\s)-\a(\x,s_0,\t))\cdot(\s-\t)>0,\quad \s\neq \t;\label{us2}
\end{equation}
\begin{equation}\label{us3}
\a(\x,s_0,\s)\cdot\s\geq \overline{a} \, \P(\x,\s)-\phi(\x).
\end{equation}
Hereinafter, we use the notations
 $\P(\x,\s)=\sum\limits_{i=1}^n|s_i|^{p_i(\x)},\;$
$\P'(\x,\s)=\sum\limits_{i=1}^n|s_i|^{p'_i(\x)},\;$ ${\bf P}(\x,\s_0,\s)=\P(\x,\s)+|s_0|^{p_0(\x)}$.

Let us note that in the articles known to the author a condition of the type \eqref{us3} appears only with $\phi=0$.

Furthermore, assume that there exist a nonnegative function $\Phi_0\in L_{1}(\Omega)$ and a continuous nondecreasing function $\widehat{b}:\mathbb{R}^+\rightarrow \mathbb{R}^+$ such that for a.a.\ $\x\in\Omega$ and all $s_0\in \mathbb{R},\;$ $\s\in\mathbb{R}^n$ the following inequalities are satisfied:
\begin{equation}\label{us4}
|b(\x,s_0,\s)|\leq \widehat{b}(|s_0|)\left(\P(\x,\s)
+\Phi_0(\x)\right);
\end{equation}
\begin{gather}
b(\x,s_0,\s)s_0\geq 0. \label{us5}
\end{gather}

Evidently, the assumption \eqref{us5} implies that for a.a.\ $\x\in\Omega$ and all $\s\in\mathbb{R}^n$ there holds
\begin{equation}\label{us5_}
b(\x,0,\s)=0.
\end{equation}

As an example, we can consider functions
$$
a_i(\x,s_0,\s)=\widehat{a}_i(|s_0|)\left(\P(\x,\s)^{1/p'_i(\x)}{\rm sign}\,s_i+\Phi_i(\x)\right)
,\quad i=1,\ldots,n,
$$
$$
b(\x,s_0,\s)=b(s_0)\P(\x,\s)^{1/q'(\x)}\Phi_0^{1/q(\x)}(\x),
$$
with a nonnegative nondecreasing odd function $b:\mathbb{R}\rightarrow \mathbb{R},\;$ $q(\cdot)\in C^+(\overline{\Omega})$ and a nonnegative function $\Phi_0\in L_{1}(\Omega)$.

By ${\L}_{\overrightarrow{\p}(\cdot)}(\Omega)$ we denote the space $L_{p_1(\cdot)}(\Omega)\times\ldots\times L_{p_n(\cdot)}(\Omega)$ with the norm
$$
\|{\v}\|_{{\L}_{\overrightarrow{\p}(\cdot)}(\Omega)}=\|{\v}\|_{\overrightarrow{\p}(\cdot)}=\|v_1\|_{p_1(\cdot)}+\ldots+\|v_n\|_{p_n(\cdot)},
 $$
 $${\v}=(v_1,\ldots,v_n)\in {\L}_{\overrightarrow{\p}(\cdot)}(\Omega).
$$
And by ${\bf L}_{\overrightarrow{\bf p}(\cdot)}(\Omega)$ we denote the space $L_{p_0(\cdot)}(\Omega)\times{\L}_{\overrightarrow{\p}(\cdot)}(\Omega)$ with the norm
$$
\|{\bf v}\|_{{\bf L}_{\overrightarrow{\bf p}(\cdot)}(\Omega)}=\|v_0\|_{p_0(\cdot)}+\|{\v}\|_{\overrightarrow{\p}(\cdot)},\quad {\bf v}=(v_0,v_1,\ldots,v_n)\in {\bf L}_{\overrightarrow{\bf p}(\cdot)}(\Omega).
$$

We will assume that
$$
\mu=f-{\rm div}\, \f,\quad f\in L_1(\Omega),\quad \f=(f_1,\dots,f_n)\in {\L}_{\overrightarrow{\p}'(\cdot)}(\Omega).
$$
Introducing a notation
$\widetilde{\a}(\x,s_0,\s)=\a(\x,s_0,\s)+\f$, we obtain from the equation \eqref{ur} that
$$
{\rm div}\,\widetilde{\a}(\x,u,\nabla
u)=|u|^{p_0(\x)-2}u+b(\x,u,\nabla u)+f.
$$
Applying the inequality \eqref{ung}, we easily see that the functions $\widetilde{\a}(\x,s_0,\s)$ also satisfy assumptions of the type \eqref{us1} -- \eqref{us3}.
Thus, we will consider the equation \eqref{ur} with
\begin{equation}\label{us6}
\mu=f,\quad f\in L_1(\Omega).
\end{equation}

Let us define the function
$$
T_k(r)=
  \begin{cases}
    k & \mbox{ for } r>k , \\
    r & \mbox{ for } |r|\le k , \\
   -k & \mbox{ for } r<-k,
  \end{cases}
$$
and introduce the notation $\langle u \rangle=\int\limits_{\Omega}u d\x$.
By $\mathring
{\mathcal{T}}_{\overrightarrow{\bf p}(\cdot)}^1(\Omega)$ we denote the set of measurable functions $u:$ $\Omega\to\RR$ such that $T_k(u)\in \mathring {W}_{\overrightarrow{\bf p}(\cdot)}^1(\Omega)$ for any $k>0$.  Let $\chi_Q$ be the indicator function of a set $Q$.
For $u\in \mathring
{\mathcal{T}}_{\overrightarrow{\bf p}(\cdot)}^1(\Omega)$ and any $k>0$ we have
\begin{equation}\label{3_0}
\nabla T_k(u)=\chi_{\{\Omega:|u|<k\}}\nabla u \in  \L_{\overrightarrow{\p}(\cdot)}(\Omega).
\end{equation}

\begin{definition}\label{D1}
	An entropy solution of the problem \eqref{ur}, \eqref{gu}, \eqref{us6} is a function
   $u\in \mathring{\mathcal{T}}_{\overrightarrow{\bf p}(\cdot)}^1(\Omega)$  such that
   \begin{enumerate}
   	\item[1)] $B(\x)=b(\x,u,\nabla u)\in L_{1}(\Omega)$;
   	\item[2)] for all $k>0$ and $\xi\in C_0^1(\Omega)$ the following inequality is satisfied:
   	\begin{equation}\label{intn}
   	\langle (b(\x,u, \nabla u) +|u|^{p_0(\x)-2}u+f(\x))T_k(u-\xi) \rangle+\langle \a(\x,u,\nabla u)\cdot\nabla T_k(u-\xi)
   	\rangle\leq 0.
   	\end{equation}
   \end{enumerate}
\end{definition}

The main result of the present work is the following theorem.
\begin{theorem}\label{t1}
	Let the assumptions \eqref{us0}--\eqref{us5} be satisfied.
	Then there exists an entropy solution of the problem \eqref{ur}, \eqref{gu}, \eqref{us6}.
\end{theorem}

\section{Preliminaries}\label{s4}

We denote by $L_{1,{\rm loc}}(\overline{\Omega})$ the space of functions $v:\Omega\rightarrow\mathbb{R}$ such that $v\in L_1(Q)$ for any bounded set $Q\subset \Omega$.
Analogously, we define the space $\L_{\overrightarrow{\p}(\cdot),{\rm loc}}(\overline{\Omega})$.

All constants appearing below in the paper are assumed to be positive.
Applying \eqref{sum}, for a.a.\ $\x\in\Omega$ and any $(s_0,\s)\in\mathbb{R}^{n+1}$ we deduce from \eqref{us1} the estimates
\begin{equation}\label{us1'}
\tag{$\ref*{us1}^\prime$}
|a_i(\x,s_0,\s)|^{p'_i(\x)}\leq \widehat{A}_i(|s_0|)
\left(\P(\x,\s)+\Psi_i(\x)\right),\quad
i=1,\ldots,n,
\end{equation}
with nonnegative functions $\Psi_i\in L_1(\Omega)$ and continuous nondecreasing functions $\widehat{A}_i:\mathbb{R}^+\rightarrow \mathbb{R}^+\setminus\{0\},\;i=1,\ldots,n$.

Applying \eqref{us1'}, we derive from \eqref{3_0} that for any $u\in \mathring{\mathcal{T}}_{\overrightarrow{\bf p}(\cdot)}^1(\Omega)$ and $k>0$,
\begin{equation}\label{axN}
  \chi_{\{\Omega:|u|< k\}}\a(\x,u, \nabla u)\in
\L_{\overrightarrow{\p}'(\cdot)}(\Omega).
\end{equation}

\begin{lemma}\label{lemma_1}
	If $u$ is an entropy solution of the problem \eqref{ur}, \eqref{gu}, then for any $k>0$ the following inequality is satisfied:
\begin{equation}\label{lemma_1_}
\int\limits_{\{\Omega : |u|<k\}}{\bf P}(\x,u,\nabla u)d\x+k\int\limits_{\{\Omega : |u|\geq k\}}|u|^{p_0(\x)-1} d\x\leq C_1 k+C_2.
\end{equation}

\end{lemma}
\begin{proof}
According to the inequality \eqref{intn}, we have for $\xi=0$,
\begin{gather*}
\int\limits_{\Omega}(|u|^{p_0(\x)-2}u + b(\x,u,\nabla u))T_k(u) d\x+\int\limits_{\{\Omega : |u|<k\}}\a(\x,u,\nabla
u)\cdot \nabla u d\x= \\=-\int\limits_{\Omega}f(\x)
T_{{k}}(u)d\x\leq k\|f\|_1.
\end{gather*}
Applying the inequalities \eqref{us3}, \eqref{us5}, we get
\begin{gather*}
k\int\limits_{\{\Omega : |u|\geq k\}}|u|^{p_0(\x)-1}d\x+\int\limits_{\{\Omega : |u|<k\}} |u|^{p_0(\x)}d\x+\overline{a}\int\limits_{\{\Omega : |u|<k\}} \P(\x,\nabla u)d\x \leq k\|f\|_1+\|\phi \|_1.
\end{gather*}
Hence, we obtain \eqref{lemma_1_}.
\end{proof}

\begin{lemma}\label{lemma_2}
Let $v:\Omega\rightarrow\mathbb{R}$ be a measurable function such that for all $k>0$ there holds
\begin{equation}\label{lemma_2_1}
\int\limits_{\{\Omega : |v|\geq k\}}|v|^{p_0(\x)-1} d\x \leq C_3+C_4/k.
\end{equation}
Then
\begin{equation}\label{lemma_2_2}
{\rm meas}\,\{\Omega :|v|\geq  k\}\rightarrow 0, \quad k\rightarrow \infty;
\end{equation}
\begin{equation}\label{lemma_2_3}
\forall k>0\quad|v|^{p_0(\x)-1}\chi_{\{\Omega:|v|\geq k\}}\in L_1(\Omega).
\end{equation}
\end{lemma}

\begin{proof}
The fact \eqref{lemma_2_3} is a trivial consequence of \eqref{lemma_2_1}. From \eqref{lemma_2_1} we have
\begin{gather*}
k^{p^-_0-1}{\rm meas}\{\Omega : |v|\geq k\}\leq C_2,\quad k\geq 1,
\end{gather*}
and hence we get \eqref{lemma_2_2}.
\end{proof}

\begin{remark}\label{remark_1}
	If $u$ is an entropy solution of the problem \eqref{ur}, \eqref{gu}, then from Lemmas \ref{lemma_1}, \ref{lemma_2} it follows that
\begin{equation}\label{remark_1_1}
{\rm meas}\,\{\Omega :|u|\geq  k\}\rightarrow 0, \quad k\rightarrow \infty;
\end{equation}
\begin{equation}\label{remark_1_2}\forall k>0\quad |u|^{p_0(\x)-1}\chi_{\{\Omega:|u|\geq k\}}\in L_1(\Omega).
\end{equation}
Moreover,
\begin{equation}\label{remark_1_3}
\forall k>0\quad |u|^{p_0(\x)}\chi_{\{\Omega:|u|< k\}}\in L_1(\Omega).
\end{equation}
\end{remark}

Let $Q\subset\mathbb{R}^n$ be an arbitrary domain.

\begin{lemma}\label{lemma_4}
Let $v^j,\, j\in \mathbb{N}$, and $v$ be functions from $L_{p(\cdot)}(Q)$ such that
$\{v^j\}_{j\in \mathbb{N}}$ is bounded in
$L_{p(\cdot)}(Q)$ and
$$
v^j\rightarrow
v \quad\mbox{a.e.\ in}\quad Q,\quad j\rightarrow\infty.
$$
Then
$$
v^j\rightharpoonup v
\quad\mbox{weakly in}\quad L_{p(\cdot)}(Q),\quad j\rightarrow\infty.
$$
\end{lemma}
The proof of Lemma \ref{lemma_4} for bounded domains is given in \cite{Benboub_2011} and it remains valid for unbounded domains, too.

\begin{lemma}\label{lemma_10}
Let $g^j,\, j\in \mathbb{N}$, and $g$ be functions from $L_1(Q)$ such that
$$
g^j\rightarrow g \quad\mbox{stronly in}\quad L_1(Q),\quad j\rightarrow\infty,
$$
and let $v^j,\, j\in \mathbb{N}$, and $v$ be measurable functions in $Q$ such that
$$
v^j\rightarrow v \quad\mbox{a.e.\ in}\quad Q,\quad j\rightarrow\infty;
$$
$$
|v^j|\leq |g^j|,\quad j\in \mathbb{N},\quad |v|\leq |g|\quad\mbox{a.e.\ in}\quad Q.
$$
Then
$$
v^j\rightarrow v \quad\mbox{stronly in}\quad L_1(Q),\quad j\rightarrow\infty.
$$
\end{lemma}

\begin{lemma}\label{lemma_6}
 If $u$ is an entropy solution of the problem \eqref{ur}, \eqref{gu}, then the inequality \eqref{intn} holds true for any function $\xi\in\mathring
{W}_{\overrightarrow{\bf p}(\cdot)}^1(\Omega)\cap L_{\infty}(\Omega)$.
\end{lemma}

\begin{proof}
	By the definition of the space $\mathring
{W}_{\overrightarrow{\bf p}(\cdot)}^1(\Omega)\cap L_{\infty}(\Omega)$ there exists a sequence $\{\xi^m\}_{m\in\mathbb{N}}\in C_0^{\infty}(\Omega)$ bounded in $L_{\infty}(\Omega)$ and satisfying \begin{equation}\label{lem_6(0)}
  \nabla \xi^m\rightarrow \nabla\xi\;\text{ in }\;\L_{\overrightarrow{\p}(\cdot)}(\Omega),\quad
  \xi^m\rightarrow \xi\;\text{ in }\;L_{{p}_0(\cdot)}(\Omega),\quad m\rightarrow\infty.
\end{equation}
Thus, we get the convergences $\xi^m\rightarrow \xi$, $\nabla \xi^m\rightarrow \nabla\xi$ in $L_{1,{\rm loc}}(\overline{\Omega})$ as $m\rightarrow\infty$, and hence we can extract a subsequence (denoted by the same indexes) such that $\xi^m\rightarrow \xi,$   $\nabla \xi^m\rightarrow \nabla\xi$  a.e.\ in $\Omega$ as $m\rightarrow\infty$.
Therefore, for any $k>0$ there are the convergences
\begin{equation}\label{lem_6(1)}
T_k(u-\xi^m)\rightarrow T_k(u-\xi),\quad\nabla T_k(u-\xi^m)\to \nabla T_k(u-\xi)\quad \mbox{a.e. \; in} \quad\Omega,\quad m\rightarrow\infty.
\end{equation}

Let $\widehat{k}=k+\sup\limits_{m\in
\mathbb{N}}(\|\xi^m\|_{\infty},\|\xi\|_{\infty})$. Then
$$
|\nabla
T_k(u-\xi^m)|\le |\nabla T_{\widehat{k}}(u)|+|\nabla \xi^m|,\quad
\x\in \Omega,\quad m\in \mathbb{N}.
$$
Since the convergent subsequence $\{\nabla \xi^m\}_{m \in \mathbb{N}}$ is bounded in
$\L_{\overrightarrow{\p}(\cdot)}(\Omega)$, we get from
\eqref{3_0} the boundedness of the norms $\|\nabla
T_k(u-\xi^m)\|_{\overrightarrow{\p}(\cdot)},\;m\in \mathbb{N}$.
Applying \eqref{lem_6(1)} and using Lemma \ref{lemma_4}, for any $k>0$ we have
  \begin{equation}\label{lem_6(2)}
  \nabla T_k(u-\xi^m)\rightharpoonup \nabla T_k(u-\xi)\quad \text{in} \quad\L_{{\overrightarrow{\p}(\cdot)}}(\Omega),\quad m\rightarrow\infty.
  \end{equation}

Let us now pass to the limit as $m\rightarrow\infty$ in the inequality
\begin{gather}
\int\limits_{\Omega} (b(\x,u,\nabla u)+f)T_k(u-\xi^m)d\x+\int\limits_{\Omega} \a(\x,u,\nabla u)\cdot\nabla
T_k(u-\xi^m)d\x+\nonumber\\+\int\limits_{\Omega} |u|^{p_0(\x)-2}uT_k(u-\xi^m)d\x\leq 0.\label{lem_6(3)}
\end{gather}
Since $b(\x,u,\nabla u),\;f\in L_{1}(\Omega)$ (see Definition \ref{D1}), applying \eqref{lem_6(1)} and using the Lebesgue theorem, we can pass to the limit as $m\rightarrow
\infty$ in the first summand in \eqref{lem_6(3)}.
Since
$\a(\x,u,\nabla u)\chi_{\{\Omega:|u|<\widehat{k}\}}\in \L_{\overrightarrow{\p}'(\cdot)}(\Omega)$ (see \eqref{axN}), we apply \eqref{lem_6(2)} to establish that the second summand in \eqref{lem_6(3)} also has a limit as $m\rightarrow \infty$. Let $k_1>\sup\limits_{m\in
\mathbb{N}}\|\xi^m\|_{\infty}$.
Finally, let us split the third summand in \eqref{lem_6(3)} as
\begin{gather*}
\int\limits_{\Omega} |u|^{p_0(\x)-2}uT_k(u-\xi^m)d\x=\int\limits_{\{\Omega:|u-\xi^m|<k_1\}}|u|^{p_0(\x)-2}uT_k(u-\xi^m)d\x+\\ +\int\limits_{\{\Omega:|u-\xi^m|\geq k_1 \}} |u|^{p_0(\x)-2}u k\,{\rm sign}\,(u-\xi^m)d\x=I_1^m+I_2^m.
\end{gather*}
In view of \eqref{remark_1_3},
$$
|u|^{p_0(\x)-1}\chi_{\{\Omega:|u-\xi^m|<k_1\}}|T_k(u-\xi^m)|\leq |u|^{p_0(\x)-1}\chi_{\{\Omega:|u|<\widehat{k}_1\}}(|u|+|\xi^m|)\in L_1(\Omega),\quad m\in \mathbb{N}.
$$
Thus, we use \eqref{lem_6(0)} and deduce from Lemma \ref{lemma_10} that
 $$\lim\limits_{m\to \infty}I_1^m\rightarrow\int\limits_{\{\Omega:|u-\xi|<k_1\}}|u|^{p_0(\x)-2}uT_k(u-\xi)d\x.$$

Due to \eqref{remark_1_2},
$$
|u|^{p_0(\x)-1}\chi_{\{\Omega:|u-\xi^m|\geq k_1\}}\leq |u|^{p_0(\x)-1}\chi_{\{\Omega:|u|\geq\widetilde{k}_1\}}\in L_1(\Omega),\quad m\in \mathbb{N},
$$
$\widetilde{k}_1=k_1-\sup\limits_{m\in
\mathbb{N}}\|\xi^m\|_{\infty}$.
Hence,
we obtain from the Lebesgue theorem that
$$
\lim\limits_{m\to \infty}I_2^m\rightarrow\int\limits_{\{\Omega:|u-\xi|\geq k_1\}} |u|^{p_0(\x)-2}u k\,{\rm sign}\,(u-\xi)d\x.
$$
Therefore, passing to the limit in \eqref{lem_6(3)}, we derive the inequality \eqref{intn}.
\end{proof}

\begin{lemma}\label{lemma_3}
Let $u$ be an entropy solution of the problem \eqref{ur}, \eqref{gu}. Then for all $k>0$ there holds
\begin{equation}\label{lemma_3_}
\lim_{h\to \infty}\int\limits_{\{\Omega : h\leq |u|<k+h\}}\P(\x,\nabla u) d\x=0.
\end{equation}
\end{lemma}

\begin{proof}
Consider the function $T_{k,h}(r)=T_k(r-T_h(r))$. Evidently,
\begin{equation*}T_{k,h}(r)=
  \begin{cases}
    0 & \mbox{ for }|r|<h , \\
        r-h\,{\rm sign}\, r & \mbox{ for }h\leq|r|< k+h, \\
    k\,{\rm sign}\, r & \mbox{ for }|r|\geq k+h.
  \end{cases}
  \end{equation*}
Denoting $\xi=T_{h} (u)$ in \eqref{intn} and taking into account \eqref{us5}, we get
\begin{gather*}\int\limits_{\{\Omega : h\leq|u|<k+h\}} \a(\x,u,\nabla u)\cdot\nabla u d\x+k \int\limits_{\{\Omega :|u|\geq k+h\}}\left(|b(\x,u,\nabla u)|+|u|^{p_0(\x)-1}\right) d\x+\\
+\int\limits_{\{\Omega : h\leq|u|< k+h\}}\left(b(\x,u,\nabla u)+|u|^{p_0(\x)-2}u\right)(u-h\,{\rm sign}\,u) d\x\leq k\int\limits_{\{\Omega : |u|\geq h\}}|f|d\x.\nonumber
\end{gather*}
In view of \eqref{us5}, for $h\leq |u|$ the following inequality is satisfied:
$$
(b(\x,u,\nabla u)+|u|^{p_0(\x)-2}u)(u-h\,{\rm sign}\,u)\ge 0.
$$
Combining the last two inequalities, we deduce that
\begin{gather*}\int\limits_{\{\Omega : h\leq|u|<k+h\}} \a(\x,u,\nabla u)\cdot\nabla u d\x+\\
+k\int\limits_{\{\Omega : |u|\geq k+h\}}\left(|b(\x,u,\nabla u)|+|u|^{p_0(\x)-1}\right) d\x
\leq k\int\limits_{\{\Omega : |u|\geq h\}}|f|d\x.\nonumber
\end{gather*}
Applying \eqref{us3}, we obtain for any $k>0$ that
$$
\overline{a}\int\limits_{\{\Omega : h\leq |u|<k+h\}}\P(\x,\nabla u) d\x\leq
 k\int\limits_{\{\Omega : |u|\geq h\}}|f|d\x+\int\limits_{\{\Omega :  h\leq |u|<k+h\}}|\phi|d\x.
$$
Thus, noting that $f, \phi \in L_1(\Omega)$ and taking into account  \eqref{remark_1_1}, we derive the relation \eqref{lemma_3_}.
\end{proof}

\begin{remark}\label{remark_2}
	To avoid cumbersome arguments, instead of writing like ``from the sequence $\{v^j\}_{j \in \mathbb{N}}$ we can extract a subsequence (denoted by the same indexes) convergent a.e.\ in $\Omega$ as $j\to \infty$'' we will write simply ``the sequence $\{v^j\}_{j \in \mathbb{N}}$ converges along a subsequence a.e.\ in $\Omega$ as $j\to \infty$''.
	Accordingly, we will use the term ``converges weakly along a subsequence'', etc.
\end{remark}

\begin{lemma}\label{lemma_5}
Let $v^j,\, j\in \mathbb{N}$, and $v$ be functions from $L_{p(\cdot)}(Q)$ such that
$$
v^j\rightarrow v \quad\mbox{a.e.\ in}\quad Q,\quad j\rightarrow\infty;
$$
$$
|v^j|^{p(\x)}\leq h\in L_1(Q),\quad j\in \mathbb{N}.
$$
Then
$$
v^j\rightarrow v \quad\mbox{stronly in}\quad L_{p(\cdot)}(Q),\quad j\rightarrow\infty.
$$
\end{lemma}
The validity of Lemma \ref{lemma_5} follows from the Lebesgue theorem.

\begin{lemma}\label{lemma_7}
 Let the assumptions \eqref{us1}--\eqref{us3} be satisfied in $Q$, and for some fixed $k>0$ there hold
 \begin{equation}\label{lemma_7_1}
\v^j\rightharpoonup \v\quad  \mbox{in}\quad \L_{\overrightarrow{\p}(\cdot)}(Q),\quad j\rightarrow \infty,
\end{equation}
\begin{equation}\label{lemma_7_2}
T_k(u^j)\rightarrow
 T_k (u)\quad\mbox{a.e.\ in}\quad Q,\quad j\rightarrow\infty,
 \end{equation}
\begin{equation}
\lim_{j\to \infty}\int\limits_{Q}q^j(\x)d\x=0,\label{lemma_7_3}
\end{equation}
\begin{equation}
 q^j(\x)=(\a(\x,T_k(u^j),\v^j)-\a(\x,T_k(u^j),\v))\cdot(\v^j-\v).\label{lema_7_4}
 \end{equation}
Then, along a subsequence,
\begin{equation}\label{conf_1(4)}
\v^j\rightarrow \v\quad \mbox {a.e.\ in}\quad Q,\quad j\rightarrow \infty,
\end{equation}
\begin{equation}\label{lemma_7_}\v^j\rightarrow \v\quad \mbox{strongly in}\quad  \L_{\overrightarrow{\p}(\cdot)}(Q),\quad j\rightarrow \infty.
\end{equation}
\end{lemma}
\begin{proof}
 The convergence \eqref{conf_1(4)} is established analogously as in the proof of \cite[Assertion 2]{KKM}.
 Apparently, the first statement of this kind is Lemma 3.3 from the work \cite{LL}.

From \eqref{lemma_7_2}, \eqref{conf_1(4)} and the continuity of
$\a(\x,s_0,\s)$ with respect to ${\bf s}=(s_0,\s)$ it follows that
 \begin{equation*}
 \a(\x,T_k(u^j),\v^j)\rightarrow \a(\x,T_k(u),\v)\quad  \mbox{a.e.\ in} \quad Q,\quad j\rightarrow\infty.
 \end{equation*}
From the convergence \eqref{lemma_7_1} and in view of \eqref{st2} we have the estimate
\begin{equation}\label{conf_1(2)}
\|\P(\x,\v^j)\|_{1,Q}\leq C_3,\quad j\in \mathbb{N}.
\end{equation}
From \eqref{conf_1(2)} and \eqref{us1'} we obtain the boundedness of $\a(\x,T_k(u^j),\v^j)$ in $\L_{\overrightarrow{\p}'(\cdot)}(Q)$. Using Lemma \ref{lemma_4}, we get the weak convergence
\begin{gather}
\a(\x,T_k(u^j),\v^j)\rightharpoonup \a(\x,T_k(u),\v)\quad  \mbox{in} \quad \L_{\overrightarrow {\p}'(\cdot)}(Q),\quad j\rightarrow\infty.\label{conf_1(6)}
\end{gather}

 From \eqref{lemma_7_2} we get
$$
\a(\x,T_k(u^j),\v)\rightarrow \a(\x,T_k(u), \v)\quad  \mbox{a.e.\ in}\quad  Q,\quad j\rightarrow\infty,
$$
and from \eqref{us1'} we obtain the estimates
$$
\P'(\x,\a(\x,T_k(u^j),\v))\leq
\widehat{A}_+(k)\Psi^n(\x)+\widehat{A}^n(k)\P(\x,\v))\in L_1(Q),\quad j\in \mathbb{N},
$$
where $\Psi^n(\x)=\sum\limits_{i=1}^n\Psi_i(\x),\;\widehat{A}^n=\sum\limits_{i=1}^n\widehat{A}_i,\;\widehat{A}_+=
\max\limits_{i=\overline{1,n}}\widehat{A}_i$.
Thus, by Lemma \ref{lemma_5} we have the convergence
\begin{equation}\label{conf_1(7)}
\a(\x,T_k(u^j),\v)\rightarrow \a(\x,T_k(u),\v)
\,\text{  strongly in  }\, \L_{\overrightarrow{\p}'(\cdot)}(Q),\quad j\rightarrow\infty.
\end{equation}

Let us denote $y^j=\a(\x,T_k(u^j),\v^j)\cdot \v^j,\;y=\a(\x,T_k(u),\v)\cdot \v$.
Using \eqref{lemma_7_1}, \eqref{lemma_7_3}, \eqref{conf_1(6)}, \eqref{conf_1(7)}, we establish the convergence
\begin{equation}\label{conf_1(8)}
\int\limits_{Q}y^jd\x\rightarrow \int\limits_{Q}y d\x,\quad j\rightarrow\infty.
\end{equation}
Using \eqref{lemma_7_2}, \eqref{conf_1(4)}, we deduce the convergence
\begin{equation}\label{conf_1(9)}
y^j\rightarrow y  \quad  \mbox{a.e.\ in} \quad Q,\quad j\rightarrow\infty.
\end{equation}

Applying Fatou's Lemma, \eqref{us3}, \eqref{conf_1(4)} and \eqref{conf_1(9)}, we get the inequality
$$
\int\limits_{Q}2(y+\phi) d\x\leq \lim_{j\to\infty}\inf \int\limits_{Q} (y^j+y+2\phi-\overline{a}2^{1-p_+^+}\P(\x, \v^j-\v)d\x,
$$
where $p_+^+=\max\limits_{i=\overline{1,n}}p^+_i$.
Taking into account \eqref{conf_1(8)}, we obtain
$$
0\leq -\lim_{j\to\infty}\sup\int\limits_{Q} \P(\x,\v^j-\v)d\x.
$$
Thus, we have
$$
0\leq \lim_{j\to\infty}\inf\int\limits_{Q} \P(\x,\v^j-\v)d\x\leq\lim_{j\to\infty}\sup\int\limits_Q \P(\x,\v^j-\v)d\x\leq 0,
$$
and hence
$$
\int\limits_{Q} \P(\x,\v^j-\v)d\x\rightarrow 0, \quad j\rightarrow\infty.
$$
Consequently,
\begin{equation}\label{conf_1(10)}
\v^j\rightarrow \v \quad \text{in}\quad \L_{\overrightarrow{\p}(\cdot)}(Q), \quad j\rightarrow\infty,
\end{equation}
i.e., the convergence \eqref{lemma_7_} is proved.
\end{proof}

\begin{lemma}\label{lemma_8}
Let functions $v^j,\, j\in \mathbb{N}$, and $v\in L_{\infty}(Q)$
be such that $\{v^j\}_{j\in \mathbb{N}}$ is bounded in $L_{\infty}(Q)$ and
$$
v^j\rightarrow v \quad\mbox{a.e.\ in}\quad Q,\quad j\rightarrow\infty.
$$
Then
$$
v^j\stackrel{*}\rightharpoonup v \quad\mbox{weakly in}\quad L_{\infty}(Q),\quad j\rightarrow\infty.
$$

If, moreover, $h^j,\, j\in \mathbb{N}$, and $h$ are functions from $L_{p(\cdot)}(Q)$ such that
$$
h^j \rightarrow h \quad\mbox{strongly in}\quad L_{p(\cdot)}(Q),\quad j\rightarrow\infty,
$$
then
$$
v^jh^j\rightarrow v h \quad\mbox{strongly in}\quad L_{p(\cdot)}(Q),\quad j\rightarrow\infty.
$$
\end{lemma}
The proof of Lemma \ref{lemma_8} follows from the Lebesgue theorem.
Below, we will use the Vitali theorem in the following form (see \cite[
Chapter III, \S 6, Theorem 15]{Danf_Shvarc}).
\begin{lemma}\label{lemma_9}
Let $v^j,\, j\in \mathbb{N}$, and $v$ be measurable functions in a bounded domain $Q$ such that
$$
v^j\rightarrow v \quad\mbox{a.e.\ in}\quad Q,\quad j\rightarrow\infty,
$$
and integrals
$$
\int\limits_Q|v^j(\x)|d\x,\quad j\in\mathbb{N}
$$
are uniformly absolutely continuous.
Then
$$
v^j\rightarrow v \quad\text{strongly in}\quad L_1(Q),\quad j\rightarrow\infty.
$$
\end{lemma}

\section{Existence of a solution}\label{s5}

\textit{Proof of Theorem \ref{t1}.}

\subsection{}\label{s5.1}
Let
 $$
 f^m(\x)=\frac{f(\x)}{1+|f(\x)|/m}\chi_{\Omega(m)},\quad  \Omega(m)=\{\x\in \Omega:|\x|<m\}.
 $$
 It is not hard to show that
 \begin{equation}\label{theorem_2(1-0)}
f^m\rightarrow f \quad \mbox {in}  \quad L_1(\Omega),\quad m\rightarrow\infty,
\end{equation}
 and, moreover,
\begin{equation}\label{theorem_2(1-1)}
|f^m(\x)|\leq |f(\x)|,\quad |f^m(\x)|\leq m\chi_{\Omega(m)},\quad \x\in \Omega,\quad m\in\mathbb{N}.
\end{equation}

Consider the equations
\begin{equation}\label{theorem_2(1-2)}
 {\rm div}\;\a^m(\x,u,\nabla u)=a_0^m(\x,u,\nabla u),\quad \x\in \Omega,\quad m\in\mathbb{N},
\end{equation}
where
$$
\a^m(\x,s_0,\s)=\a(\x,T_m(s_0),\s),\quad a_0^m(\x,s_0,\s)=b^m(\x,s_0,\s)+|s_0|^{p_0(\x)-2}s_0+f^m(\x),
$$
and
$$
\a^m(\x,s_0,\s)=(a^m_1(\x,s_0,\s),\dots,a^m_n(\x,s_0,\s)),\quad b^m(\x,s_0,\s)=\frac{b(\x,s_0,\s)}{1+|b(\x,s_0,\s)|/m}\chi_{\Omega(m)}.
$$
Evidently,
\begin{equation}
|b^m(\x,s_0,\s)|\leq |b(\x,s_0,\s)|,\quad |b^m(\x,s_0,\s)|\leq m\chi_{\Omega(m)},\quad \x\in \Omega,\quad (s_0,\s)\in \mathbb{R}^{n+1}.\label{theorem_2(1-3)}
\end{equation}
Moreover, applying \eqref{us5}, \eqref{us5_}, we get
\begin{equation}\label{theorem_2(1-4)}
b^m(\x,\s,s_0)s_0\geq 0,\quad \x\in \Omega,\quad  (s_0,\s)\in \mathbb{R}^{n+1}.
\end{equation}
$$
b^m(\x,0,\s)=0,\quad \x\in \Omega,\quad \s\in \mathbb{R}^n.
$$

Let us define operators ${\bf A}^m:\mathring{W}^{1}_{\overrightarrow{\bf p}(\cdot)}(\Omega)\rightarrow{W}^{-1}_{\overrightarrow{\bf p}'(\cdot)}(\Omega),
\;{\bf A}^m=\A^m+{A}^m_0$ for any $u,v\in \mathring{W}^{1}_{\overrightarrow{\bf p}(\cdot)}(\Omega)$ by
\begin{gather}\label{theorem_2(1-6)}
<{\A}^m(u),v>=\langle \a^m(\x,u, \nabla u)\cdot\nabla
v\rangle,\quad
<{A}^m_0(u),v>=\langle a^m_0(\x,u,\nabla u)v\rangle.
\end{gather}

A generalized solution of the problem \eqref{theorem_2(1-2)}, \eqref{gu} is a function $u\in\mathring{W}^{1}_{\overrightarrow{\bf p}(\cdot)} ({\Omega})$ which satisfies the integral identity
\begin{equation}\label{theorem_2(1-5-6)}
<{\bf A}^m(u),v>= 0
\end{equation}
for any $v\in \mathring{W}^{1}_{\overrightarrow{\bf p}(\cdot)}(\Omega)$.
By Theorem \ref{t2} (see Appendix below) for any $m\in\mathbb{N}$ there exists a generalized solution $u^m\in\mathring{W}^{1}_{\overrightarrow{\bf p}(\cdot)}(\Omega)$ of the problem \eqref{theorem_2(1-2)}, \eqref{gu}.
Thus, for any function $v\in \mathring{W}^{1}_{\overrightarrow{\bf p}(\cdot)}(\Omega)$ there holds
\begin{equation}\label{theorem_2(1-5)}
\langle (f^m(\x)+b^m(\x,u^m,\nabla u^m)+|u^m|^{p_0(\x)-2}u^m)v\rangle+\langle \a(\x,T_m (u^m), \nabla u^m)\cdot\nabla
v\rangle= 0.
\end{equation}

\subsection{}\label{s5.2}
In this step, we establish a priori estimates for the sequence $\{u^m\}_{m\in \mathbb{N}}$.

Let $v=T_{k,h} (u^m)=T_k(u^m-T_h(u^m)),\;h,k>0,$ in \eqref{theorem_2(1-5)}. Taking into account \eqref{theorem_2(1-4)}, we get
\begin{gather}
\int\limits_{\{\Omega : h\leq|u^m|<k+h\}} \a(\x,T_m(u^m),\nabla u^m)\cdot\nabla u^m d\x+\nonumber\\
+k \int\limits_{\{\Omega :|u^m|\geq k+h\}}\left(|b^m(\x,u^m,\nabla u^m)|+|u^m|^{p_0(\x)-1}\right) d\x+\label{theorem_2(2-1)}\\
+\int\limits_{\{\Omega : h\leq|u^m|< k+h\}}\left(b^m(\x,u^m,\nabla u^m)+|u^m|^{p_0(\x)-2}u^m\right)(u^m-h\,{\rm sign}\,u^m) d\x\leq\nonumber \\\leq k\int\limits_{\{\Omega : |u^m|\geq h\}}|f^m|d\x.\nonumber
\end{gather}
In view of $(\ref{theorem_2(1-4)})$, for $h\leq |u^m|$ the following inequality is satisfied:
$$
(b^m(\x,u^m,\nabla u^m)+|u^m|^{p_0(\x)-2}u^m)(u^m-h\,{\rm sign}\,u^m)\ge 0.
$$
Applying now \eqref{theorem_2(1-1)}, we deduce from \eqref{theorem_2(2-1)} that
\begin{gather}
\int\limits_{\{\Omega : h\leq|u^m|<k+h\}} \left(\a(\x,T_m(u^m),\nabla u^m)\cdot\nabla u^m+\phi(\x)\right)  d\x+\nonumber\\
+k\int\limits_{\{\Omega : |u^m|\geq k+h\}}\left(|b^m(\x,u^m,\nabla u^m)|+|u^m|^{p_0(\x)-1}\right) d\x\leq\label{theorem_2(2-3)}\\
\leq \int\limits_{\{\Omega : |u^m|\geq h\}}(k|f|+\phi)d\x\leq k\|f\|_1+\|\phi\|_1,\quad m\in\mathbb{N}.\nonumber
\end{gather}

Now we take $T_k (u^m),\;k>0,$ as a test function in \eqref{theorem_2(1-5)}.
Due to \eqref{theorem_2(1-1)}, \eqref{theorem_2(1-4)}, we get
\begin{gather*}\int\limits_{\{\Omega : |u^m|<k\}} \a(\x,T_m(u^m),\nabla u^m)\cdot\nabla u^m d\x+\\ +k\int\limits_{\{\Omega : |u^m|\geq k\}}\left(|b^m(\x,u^m,\nabla u^m)|
+|u^m|^{p_0(\x)-1}\right) d\x+\\
+\int\limits_{\{\Omega : |u^m|< k\}}|u^m|^{p_0(\x)} d\x\leq  k\|f\|_1.
\end{gather*}
Therefore, using the inequality \eqref{us3}, we obtain
\begin{gather}\int\limits_{\{\Omega : |u^m|<k\}}\P(\x,\nabla u^m) d\x+ k\int\limits_{\{\Omega : |u^m|\geq k\}}\left(|b^m(\x,u^m,\nabla u^m)|+|u^m|^{p_0(\x)-1}\right)  d\x+\nonumber\\
+\int\limits_{\{\Omega : |u^m|< k\}}|u^m|^{p_0(\x)} d\x\leq kC_1+C_2 ,\quad m\in\mathbb{N}.\label{theorem_2(2-4)}
\end{gather}

From the estimate \eqref{theorem_2(2-4)} we have
\begin{gather}
\int\limits_{\Omega}|T_k(u^m)|^{p_0(\x)}d\x=\int\limits_{\{\Omega : |u^m|<k\}}|u^m|^{p_0(\x)}d\x+\int\limits_{\{\Omega : |u^m|\geq k\}}k^{p_0(\x)}d\x\leq\nonumber\\ \leq\int\limits_{\{\Omega : |u^m|<k\}}|u^m|^{p_0(\x)}d\x+k\int\limits_{\{\Omega : |u^m|\geq k\}}|u^m|^{p_0(\x)-1}d\x\leq kC_1+C_2,\quad m\in\mathbb{N}.
\label{theorem_2(2-5)}
\end{gather}
Moreover, from \eqref{theorem_2(2-4)} we get the following estimate:
\begin{equation}\label{theorem_2(2-6)}\int\limits_{\{\Omega : |u^m|<k\}}\P(\x,\nabla u^m) d\x=\int\limits_{\Omega}\P(\x,\nabla T_k (u^m))d\x\leq C_1k+C_2, \quad m\in\mathbb{N}.
\end{equation}

Combining \eqref{theorem_2(1-3)}, \eqref{us4}, \eqref{theorem_2(2-6)}, we derive
\begin{gather}
\int\limits_{\{\Omega : |u^m|<k\}}|b^m(\x,u^m,\nabla u^m)|d\x\leq  \widehat{b}(k)\int\limits_{\Omega} \left(\P(\x,\nabla T_k
(u^m))+\Phi_0(\x)\right)d\x\leq\nonumber\\
\leq \widehat{b}(k)(C_1k+C_2+\|\Phi_0\|_1)=C_3(k),\quad m\in\mathbb{N}.\label{theorem_2(2-7)}
\end{gather}
From \eqref{theorem_2(2-4)}, \eqref{theorem_2(2-7)} we get the estimate
\begin{equation}\label{theorem_2(2-8)}
\|b^m(\x,u^m,\nabla u^m)\|_{1} \leq C_4(k),\quad m\in\mathbb{N}.
\end{equation}

\subsection{}\label{s5.3}
Due to \eqref{theorem_2(2-4)}, we obtain from Lemma \ref{lemma_2} that
\begin{equation}\label{theorem_2(3-1)}
{\rm meas}\,\{\Omega : |u^m|\geq\rho\}\rightarrow 0 \quad \mbox {uniformly with respect to}\; m\in\mathbb{N}, \quad \rho\rightarrow \infty,
\end{equation}
which is equivalent to
\begin{equation}\label{theorem_2(3-1)'}
\tag{$\ref*{theorem_2(3-1)}^\prime$}
\sup_{m\in\mathbb{N}}{\rm meas}\,\{\Omega : |u^m|\geq\rho\}=g(\rho)\rightarrow 0,  \quad \rho\rightarrow \infty.
\end{equation}

Let us establish the following convergence along a subsequence:
\begin{equation}\label{theorem_2(3-2)}
u^m\rightarrow u  \quad \mbox {a.e.\ in} \quad \Omega,\quad m\rightarrow \infty.
\end{equation}

Let $\eta_R(r)=\min(1,\max(0,R+1-r))$.
Applying \eqref{sum}, we deduce from \eqref{theorem_2(2-6)} for $R, \rho>0$ that
\begin{gather*}
\int\limits_{\Omega}\P(\x,\nabla (\eta_{R}(|\x|)T_\rho (u^m)))d\x\leq\\\leq C_5 \int\limits_{\{\Omega :|u^m|<\rho\}}\P(\x,\nabla u^m ) d\x
+C_5\int\limits_{\Omega}\P(\x,T_\rho (u^m)\nabla \eta_{R}(|\x|))d\x\leq C_6(\rho,R).
\end{gather*}
Hence, for any fixed $\rho, R>0$ we have the boundedness in $\mathring{H}_{\overrightarrow{\p}(\cdot)}^1(\Omega(R+1))$ of the set $\{\eta_{R} T_\rho (u^m)\}_{m\in\mathbb{N}}$.
By Lemma \ref{lemma_0} the space $\mathring{H}_{\overrightarrow{\bf p}(\cdot)}^{1}(\Omega(R+1))$ is compactly embedded in $L_{p_-(\cdot)}({\Omega}(R+1))$. Thus, for any fixed $\rho, R>0$ we obtain the convergence $\eta_{R}T_\rho (u^m) \rightarrow v_{\rho}$ in $L_{p_-(\cdot)}(\Omega(R+1))$ as $m\rightarrow\infty$.
This implies the convergence $T_\rho (u^m) \rightarrow v_{\rho}$ in $L_{p_-(\cdot)}(\Omega(R))$, and the convergence along a subsequence $T_\rho(u^m)\rightarrow v_{\rho}$ a.e.\ in $\Omega(R)$.

By Egorov's theorem we can chose a set $E_{\rho}\subset\Omega(R)$ such that ${\rm meas}\,E_{\rho}<1/\rho$ and   \begin{equation}\label{theorem_2(3-3)}
T_\rho (u^m) \rightarrow v_{\rho}\; \mbox{uniformly on }\;
  \Omega(R)\setminus E_{\rho},\quad m\rightarrow\infty.
\end{equation}
Consider the set
$$
\Omega^{\rho}(R)=\{\x\in\Omega(R)\setminus E_{\rho}: |v_{\rho}(\x)|\ge\rho-1\}.
$$
In view of the uniform convergence \eqref{theorem_2(3-3)}, there exists $m_0\in\mathbb{N}$ such that for any $m\geq m_0$
the inequality $|T_\rho (u^m(\x))|\geq\rho-2$ is satisfied on $\Omega^{\rho}(R)$, which implies $|u^m(\x)|\geq \rho-2.$

Evidently,
\begin{equation}\label{theorem_2(3--4)}
{\rm meas}\,\Omega^{\rho}(R)\leq\sup_m{\rm meas}\{\Omega:|u^m(\x)|\geq \rho-2\}=g(\rho-2).
\end{equation}
Due to \eqref{theorem_2(3-1)'}, $g(\rho)\to0$ as $\rho\to\infty.$

Consider the set
$$
\Omega_{\rho}(R)=\{\x\in\Omega(R)\setminus E_{\rho}:
  |v_{\rho}(\x)|<\rho-1\}.
$$
According to \eqref{theorem_2(3--4)}, we have
\begin{equation}\label{theorem_2(3--5)}
{\rm meas}\,\Omega_{\rho}(R)>{\rm meas}\,\Omega(R)-1/\rho-g(\rho-2).
\end{equation}

The uniform convergence \eqref{theorem_2(3-3)} implies that $|T_\rho (u^m(\x))|<\rho$ on  $\Omega_{\rho}(R)$ for $m\geq m_0$. Therefore,
$$
u^m\rightarrow v_{\rho}
  \quad \mbox{uniformly on} \quad \Omega_{\rho}(R),\quad m\rightarrow\infty.
$$
Thus, in view of \eqref{theorem_2(3--5)}, using the diagonalisation argument with respect to $\rho\in \mathbb{N}$, it is not hard to obtain that $v_{\rho}$ does not depend on $\rho\; (v_{\rho}=u),$
and the convergence
$$
u^m\rightarrow u  \quad \mbox{a.e.\ in} \quad \Omega(R),\quad m\rightarrow\infty
$$
holds true. Then, by the diagonalisation argument with respect to $R\in \mathbb{N}$, we establish the convergence \eqref{theorem_2(3-2)}.

It follows from \eqref{theorem_2(3-2)} that for any $k>0$,
\begin{equation}\label{theorem_2(3-7)}
T_k(u^m)\rightarrow T_k(u)  \quad \mbox {a.e.\ in} \quad \Omega,\quad m\rightarrow \infty.
\end{equation}

Let us prove that
\begin{equation}\label{theorem_2(3-8)}
|u^m|^{p_0(\x)-2}u^m \to  |u|^{p_0(\x)-2}u \quad\mbox{in} \quad L_{1,{\rm loc}}(\overline{\Omega}),\quad m\rightarrow \infty.
\end{equation}

From the convergence \eqref{theorem_2(3-2)} we have
\begin{equation}\label{theorem_2(3-9)}
|u^m|^{p_0(\x)-2}u^m \to |u|^{p_0(\x)-2}u
\quad\mbox{a.e.\ in} \quad \Omega,  \quad m\to
\infty.
\end{equation}

From \eqref{theorem_2(2-3)} with $k=1$ and any $h>0$ we get
 \begin{gather*}\int\limits_{\{\Omega : h\leq |u^m|<1+h\}}\left(\a(\x,T_m(u^m),\nabla u^m)\cdot\nabla u^m+\phi(\x)\right) d\x+\\+\int\limits_{\{\Omega :|u^m|\geq h+1\}}\left(|b^m(\x,u^m,\nabla u^m)|+|u^m|^{p_0(\x)-1}\right)   d\x\leq \int\limits_{\{\Omega :|u^m|\geq h\}}(|f|+\phi)d\x,\quad m\in \mathbb{N}.\end{gather*}
Noting that $f,\phi\in L_1(\Omega)$ and the integral in the right-hand side of the last inequality is absolutely continuous, and recalling \eqref{theorem_2(3-1)}, we see that for any $\varepsilon>0$ there exists a sufficiently large $h(\varepsilon)>1$ such that
\begin{gather}\label{theorem_2(3-10)}
\int\limits_{\{\Omega : h-1\leq |u^m|<h\}}\left(\a(\x,T_m(u^m),\nabla u^m)\cdot\nabla u^m+\phi(\x)\right) d\x+\\+\int\limits_{\{\Omega :|u^m|\geq h\}}\left(|b^m(\x,u^m,\nabla u^m)|+|u^m|^{p_0(\x)-1}\right)   d\x <\frac{\varepsilon}{2},\quad m\in\mathbb{N}.\nonumber
\end{gather}

Let $Q$ be an arbitrary bounded subset of $\Omega$. Then for any measurable set $E\subset Q$ there holds
\begin{equation}\label{theorem_2(3-11)}
\int\limits_E |u^m|^{p_0(\x)-1}d\x\leq \int\limits_{ \{E:|u^m|<h\}}|u^m|^{p_0(\x)-1}d\x+\int\limits_{\{\Omega:|u^m|\geq h\}}|u^m|^{p_0(\x)-1}d\x.
\end{equation}
Evidently,
\begin{equation}\label{theorem_2(3-12)}
\int\limits_{\{E:|u^m|<h\}}|u^m|^{p_0(\x)-1}d\x\leq h^{p^+_0-1}\;{\rm meas}\;E<\frac{\varepsilon}{2},\quad m\in \mathbb{N},
\end{equation}
is satisfied for any $E$ such that ${\rm meas}\;E<\frac{\varepsilon}{2h^{p^+_0-1}}=\alpha(\varepsilon)$.

Combining \eqref{theorem_2(3-10)}--\eqref{theorem_2(3-12)}, we obtain
$$
\int\limits_E |u^m|^{p_0(\x)-1}d\x< \varepsilon\quad \forall\;E\quad \mbox{such that }\; {\rm meas}\;E<\alpha(\varepsilon),\quad m\in \mathbb{N}.
$$
This implies that the integrals $\int\limits_Q|u^m|^{p_0(\x)-1}d\x,\;m\in \mathbb{N},$ are uniformly absolutely continuous, and by Lemma \ref{lemma_9} there is the convergence
$$
|u^m|^{p_0(\x)-2}u^m \to  |u|^{p_0(\x)-2}u \quad\mbox{in} \quad L_{1}(Q),\quad m\rightarrow \infty.
$$
Since $Q\subset\Omega$ is arbitrary, the convergence \eqref{theorem_2(3-8)} is proved.

\subsection{}\label{s5.4}

Let us show that $T_k (u)\in \mathring {W}_{\overrightarrow{\bf p}(\cdot)}^{1}(\Omega)$ for any $k>0$.
Combining \eqref{theorem_2(2-5)}, \eqref{theorem_2(2-6)}, \eqref{st3}, for any fixed $k>0$ we get the estimate
\begin{equation}\label{theorem_2(6-0)}
\| T_k(u^m)\|_{\mathring
{W}_{\overrightarrow{\bf p}(\cdot)}^{1}(\Omega)}\leq C_7(k),\quad m\in\mathbb{N}.
\end{equation}
Since $\mathring
{W}_{\overrightarrow{\bf p}(\cdot)}^{1}(\Omega)$ is reflexive, we can extract a weakly convergent subsequence in $\mathring
{W}_{\overrightarrow{\bf p}(\cdot)}^{1}(\Omega)$ of the sequence $T_k (u^m)\rightharpoonup v_k,\; m\rightarrow \infty$,
where $v_k\in \mathring
{W}_{\overrightarrow{\bf p}(\cdot)}^{1}(\Omega)$.
The continuity of the natural map $\mathring
{W}_{\overrightarrow{\bf p}(\cdot)}^{1}(\Omega)\rightarrow {\bf L}_{\overrightarrow{\bf p}(\cdot)}(\Omega)$ yields the weak convergence
$$
 \nabla T_k (u^m)\rightharpoonup \nabla v_{k} \quad \mbox{in}\quad \L_{\overrightarrow{\p}(\cdot)}(\Omega),\quad
 T_k (u^m)\rightharpoonup  v_{k} \quad \mbox{in}\quad L_{p_0(\cdot)}(\Omega), \quad m\rightarrow \infty.
$$

Using the convergence \eqref{theorem_2(3-7)} and applying Lemma \ref{lemma_4}, we get the weak convergence
$$
 T_k (u^m)\rightharpoonup T_k(u) \quad \mbox{in}\quad L_{p_0(\cdot)}(\Omega),\quad m\rightarrow \infty.
$$
This implies that $v_k=T_k(u)\in \mathring
{W}_{\overrightarrow{\bf p}(\cdot)}^{1}(\Omega).$

Thus, we have the weak convergence
\begin{equation}\label{theorem_2(6-1)}
 \nabla T_k(u^m) \rightharpoonup \nabla T_k(u) \quad \mbox{in}\quad \L_{\overrightarrow{\p}(\cdot)}(\Omega),\quad m\rightarrow \infty.
\end{equation}

\subsection{}\label{s5.5}

In this step, we establish the strong convergence
\begin{equation}\label{theorem_2(4--0)}
\nabla T_k(u^m)\rightarrow \nabla T_k(u)\quad \mbox{in}\quad  \L_{\overrightarrow{\p}(\cdot),{\rm
loc}}(\overline{\Omega}),\quad m\rightarrow \infty.
\end{equation}

From \eqref{theorem_2(2-6)}, \eqref{us1'}, applying \eqref{st3}, for any $k>0$ we have the estimate
\begin{gather}
\|\a(\x,T_k(u^m),\nabla T_k(u^m))\|_{\overrightarrow{\p}(\cdot)}\leq C_8(k),
\quad m\in\mathbb{N}.\label{theorem_2(4--2)}
\end{gather}

Let us denote by $\varepsilon_i(m),\;m,i\in \mathbb{N},$ some functions which tend to $0$ as $m\rightarrow \infty$. Let $\varphi_k(\rho)=\rho\exp(\gamma^2\rho^2),$ where $\gamma=\frac{\widehat{b}(k)}{\overline{a}}.$

Let  $h-1>k>0,$
$$
z^m=T_k(u^m)-T_k(u),\quad  m\in \mathbb{N}.
$$
Evidently,
$$
\psi_k(\rho)=\varphi'_k(\rho)-\gamma|\varphi_k(\rho)|\geq 7/8,\quad \rho\in\mathbb{R}.
$$
This implies the inequalities
\begin{equation}\label{theorem_2(4--3)}
7/8\leq \psi_k(z^m)\leq \max_{[-2k,2k]}\psi_k(\rho)=C_9(k),\quad m\in\mathbb{N}.
\end{equation}

In view of \eqref{theorem_2(3-7)}, we get
\begin{equation}\label{theorem_2(4--4)}
\varphi_k(z^m)\rightarrow 0 \quad \mbox {a.e.\ in} \quad \Omega,\quad m\rightarrow \infty,
\end{equation}
and
\begin{equation}\label{theorem_2(4--5)}
|\varphi_k(z^m)|\leq \varphi_k(2k),\quad 1\leq\varphi_k'(z^m)\leq
\varphi_k'(2k),\quad m\in \mathbb{N}.
\end{equation}

Applying \eqref{theorem_2(4--4)}, \eqref{theorem_2(4--5)},
we use Lemma \ref{lemma_8} to get the convergence
\begin{equation}\label{theorem_2(4--6)}
|\varphi_k(z^m)| \stackrel{*}\rightharpoonup 0\quad\mbox{in}\quad
L_{\infty}(\Omega),\quad m\rightarrow \infty.
\end{equation}
By Lemma \ref{lemma_5}, for any $g\in L_{p_i(\cdot)}(Q),\;Q\subseteq \Omega,$ there is the convergence
\begin{equation}\label{theorem_2(4--8)}
g\varphi_k(z^m)\rightarrow 0\quad\mbox{in}\quad
L_{p_i(\cdot)}(Q),\quad i=1,\ldots,n,\quad m\rightarrow \infty.
 \end{equation}

Considering $\varphi_k(z^m)\eta_R(|\x|)\eta_{h-1}(|u^m|),\;R>0,$ as a test function in \eqref{theorem_2(1-5)}, we get
\begin{gather}
\int\limits_{\Omega}\a(\x,T_h(u^m),\nabla T_h(u^m))\cdot \nabla (\eta_R(|\x|) \varphi_k(z^m)\eta_{h-1}(|u^m|))d\x+\nonumber\\
+\int\limits_{\Omega}b^m(\x,u^m,\nabla u^m)\varphi_k(z^m)\eta_R(|\x|)\eta_{h-1}(|u^m|)d\x+\label{theorem_2(4--10)}\\
+\int\limits_{\Omega}|u^m|^{p_0(\x)-2}u^m\varphi_k(z^m)\eta_R(|\x|)\eta_{h-1}(|u^m|)d\x+\nonumber\\
+\int\limits_{\Omega}f^m\varphi_k(z^m)\eta_R(|\x|)\eta_{h-1}(|u^m|)d\x=I_1^{mh}+I_2^{mh}+I_3^{mh}+I_4^{mh}=0,\quad
m\geq h.\nonumber
\end{gather}

{\it Estimates for the integrals $I_2^{mh}-I_4^{mh}$.}
In view of the inequality $|u^m|^{p_0(\x)-1}\eta_{h-1}(|u^m|)\le
h^{p_0^+-1}$, the estimates \eqref{theorem_2(4--5)} and the convergence \eqref{theorem_2(4--4)}, the Lebesgue theorem implies that
\begin{equation}\label{theorem_2(4--11)}
|I_3^{mh}|\le\int\limits_{\Omega(R+1)}h^{p_0^+-1}|\varphi_k(z^m)|d\x=\varepsilon_{1}(m).
\end{equation}
Analogously, thanks to \eqref{theorem_2(1-1)}, we get
\begin{equation}\label{theorem_2(4--12)}|I_4^{mh}|\le\int\limits_\Omega
|f\varphi_k(z^m)|d\x=\varepsilon_{2}(m).
\end{equation}

It is clear that $z^mu^m\ge0 $ for $|u^m|\ge k$, and hence, in view of \eqref{theorem_2(1-4)}, we have
$$
b^m(\x,u^m,\nabla u^m)\varphi_k(z^m)\geq
0\quad \mbox {as}\quad |u^m|\geq k.
$$
Noting this fact and applying \eqref{theorem_2(1-3)}, \eqref{us4}, we have the following estimate:
\begin{gather}\label{theorem_2(4--13)}
-I_2^{mh}\le\int\limits_{\{\Omega:|u^m|<k\}}|b^m(\x,u^m,\nabla
u^m)||\varphi_k(z^m)|\eta_R(|\x|)d\x\leq\\
\leq\widehat{b}(k)\int\limits_{\Omega} \left(\P(\x,\nabla T_k
(u^m))+\Phi_0(\x)\right)|\varphi_k(z^m)|\eta_R(|\x|)d\x,\quad
m\in \mathbb{N}.\nonumber
\end{gather}
Using \eqref{us3}, we deduce that
\begin{gather}
-I_2^{mh}
\leq\frac{\widehat{b}(k)}{\overline{a}}\int\limits_{\Omega} \left(\overline{a}\Phi_0(\x)+\phi(\x)\right)|\varphi_k(z^m)|d\x+\nonumber\\
+\frac{\widehat{b}(k)}{\overline{a}}\int\limits_{\Omega}\a(\x,T_k(u^m),\nabla
T_k(u^m))\cdot
 \nabla T_k(u^m) |\varphi_k(z^m)|\eta_R(|\x|)d\x=I^m_{21}+I^m_{22}.\label{theorem_2(4--14)}
\end{gather}

In view of \eqref{theorem_2(4--6)}, we get
\begin{equation}\label{theorem_2(4--15)}
I^m_{21}=\frac{\widehat{b}(k)}{\overline{a}}\int\limits_{\Omega} \left(\overline{a}\Phi_0(\x)+\phi(\x)\right)|\varphi_k(z^m)|d\x=\varepsilon_3(m).
\end{equation}

Now, using the estimates \eqref{theorem_2(4--11)}, \eqref{theorem_2(4--12)}, \eqref{theorem_2(4--14)}, \eqref{theorem_2(4--15)}, we obtain from \eqref{theorem_2(4--10)} that
\begin{gather}
I^{mh}_5=I^{mh}_{11}-I^{m}_{22}=\int\limits_{\Omega}\a(\x,T_h(u^m),\nabla T_h(u^m))\cdot
\nabla z^m \varphi'_k(z^m)\eta_R(|\x|)\eta_{h-1}(|u^m|)d\x-\nonumber\\
-\frac{\widehat{b}(k)}{\overline{a}}\int\limits_{\Omega}\a(\x,T_k(u^m),\nabla
T_k(u^m))\cdot
 \nabla T_k(u^m) |\varphi_k(z^m)|\eta_R(|\x|)d\x\leq\label{theorem_2(4--17)}\\
\leq\varepsilon_4(m)+\int\limits_{\{\Omega:h-1\leq|u^m|<h\}}\a(\x,T_h(u^m),\nabla T_h(u^m))\cdot \nabla u^m  \eta_R(|\x|) |\varphi_k(z^m)|d\x+\nonumber\\
+\sum\limits_{i=1}^n\int\limits_{\{\Omega:|u^m|<h\}}|a_i(\x,T_h(u^m),\nabla
T_h(u^m))|  |\eta'_R(|\x|)|
|\varphi_k(z^m)|d\x=\nonumber\\=\varepsilon_4(m)+I^{mh}_{12}+I^{mh}_{13},\quad
m\geq h.\nonumber
\end{gather}
{\it An estimate for the right-hand side of \eqref{theorem_2(4--17)}.}
Using \eqref{theorem_2(4--5)}, we have
\begin{gather*}
I_{12}^{mh}=\int\limits_{\{\Omega:h-1\leq |u^m|<h\}}\a(\x,T_h(u^m),\nabla T_h(u^m))\cdot \nabla u^m  \eta_R(|\x|) |\varphi_k(z^m)|d\x\leq\\
\leq\varphi_k(2k)\sup\limits_{m\in
\mathbb{N}}\int\limits_{\{\Omega:h-1\leq|u^m|<h\}}\left(\a(\x,T_m(u^m),\nabla
u^m)\cdot \nabla u^m+\phi\right)  d\x+\\+\varphi_k(2k)\int\limits_{\{\Omega:h-1\leq|u^m|<h\}}\phi d\x.
\end{gather*}
Thanks to \eqref{theorem_2(3-10)}, \eqref{theorem_2(3-1)}, we get
\begin{equation}\label{theorem_2(4--18)}
I_{12}^{mh}\leq\varepsilon(h),\quad m\geq h,
\end{equation}
where $\varepsilon(h)\rightarrow 0$ as $h\rightarrow\infty$.

Then, using \eqref{theorem_2(4--2)} and \eqref{theorem_2(4--8)} with $g=1\in L_{p_i(\cdot)}(\Omega(R+1))$,
we obtain the following estimate:
\begin{gather}
I^{mh}_{13}=\sum\limits_{i=1}^n\int\limits_{\{\Omega:|u^m|<h\}}|a_i(\x,T_h(u^m),\nabla
T_h(u^m))|  |\eta'_R(|\x|)| |\varphi_k(z^m)|d\x\leq\nonumber
\\
\leq C_{10}(h)\sum\limits_{i=1}^n \|\varphi_k(z^m)\|_{p_i(\cdot),\Omega(R+1)}=\varepsilon_5(m).\label{theorem_2(4--19)}
\end{gather}

Combining  \eqref{theorem_2(4--17)}--\eqref{theorem_2(4--19)}, we get the inequalities
\begin{equation}\label{theorem_2(4--20)}
  I^{mh}_5\le\varepsilon(h)+\varepsilon_6(m),\quad m\geq h.
\end{equation}

{\it Estimates for the integral $I_5^{mh}$.} Performing elementary transformations, we derive the following chain of equalities:
\begin{gather*}
I^{mh}_5=\int\limits_{\Omega}\a(\x,T_h(u^m),\nabla T_h(u^m))\cdot
\nabla T_k(u^m)
\varphi'_k(z^m)\eta_R(|\x|)\eta_{h-1}(|u^m|))d\x-\\-
\int\limits_{\Omega}\a(\x,T_h(u^m),\nabla T_h(u^m))\cdot \nabla
T_k(u) \varphi'_k(z^m)\eta_R(|\x|)\eta_{h-1}(|u^m|))d\x-\\
-\frac{\widehat{b}(k)}{\overline{a}}\int\limits_{\Omega}\a(\x,T_k(u^m),\nabla
T_k(u^m))\cdot
 \nabla T_k(u^m) |\varphi_k(z^m)|\eta_R(|\x|)d\x=\\
 =\int\limits_{\Omega}\a(\x,T_k(u^m),\nabla T_k(u^m))\cdot \nabla
T_k(u^m) \psi(z^m)\eta_R(|\x|)-\\-
\int\limits_{\Omega}\a(\x,T_h(u^m),\nabla T_h(u^m))\cdot \nabla
T_k(u) \varphi'_k(z^m)\eta_R(|\x|)\eta_{h-1}(|u^m|))d\x=\\=\int\limits_{\Omega}\a(\x,T_k(u^m),\nabla T_k(u^m))\cdot \nabla
z^m \psi(z^m)\eta_R(|\x|)+\\
+\int\limits_{\Omega}\a(\x,T_k(u^m),\nabla T_k(u^m))\cdot \nabla
T_k(u) \psi(z^m)\eta_R(|\x|)-\\-\int\limits_{\Omega}\a(\x,T_h(u^m),\nabla T_h(u^m))\cdot
\nabla T_k(u) \varphi'_k(z^m)\eta_R(|\x|)\eta_{h-1}(|u^m|)d\x.
\end{gather*}
The following equality is obvious:
\begin{gather}
I^{mh}_5=\int\limits_{\Omega}\a(\x,T_k(u^m),\nabla T_k(u^m))\cdot
\nabla z^m \psi(z^m)\eta_R(|\x|)d\x
-\label{theorem_2(4--21)}\\
-\frac{\widehat{b}(k)}{\overline{a}}\int\limits_{\Omega}\a(\x,T_k(u^m),\nabla
T_k(u^m)) \cdot \nabla T_k(u)
|\varphi_k(z^m)|\eta_R(|\x|)d\x+\nonumber\\
+\int\limits_{\{\Omega:|u^m|\geq k\}}\left(\a(\x,T_k(u^m),\nabla T_k(u^m))-\eta_{h-1}
\a(\x,T_h(u^m),\nabla T_h(u^m))\right)\cdot \nabla T_k(u)
\varphi'_k(z^m)\eta_Rd\x=\nonumber\\=I_{51}^m+I_{52}^{m}+I_{53}^{mh},\quad
m\geq h.\nonumber
\end{gather}

Applying \eqref{theorem_2(4--2)}, we get
 \begin{gather*}|I^{m}_{52}|\leq \frac{\widehat{b}(k)}{\overline{a}}
\sum
\limits_{i=1}^n\int\limits_{\Omega}|a_i(\x,T_k(u^m),\nabla
T_k(u^m))|| \partial_i T_k(u)\varphi_k(z^m)|d\x\le\\
\le C_{11}(k)\sum\limits_{i=1}^n
\|\partial_iT_k(u)\varphi_k(z^m)\|_{p_i(\cdot)}.
\end{gather*}
In view of \eqref{theorem_2(4--8)} with $g=\partial_iT_k(u)\in L_{p_i(\cdot)}(\Omega)$, we have
\begin{equation}\label{theorem_2(4--22)}
I^{m}_{52}=\varepsilon_7(m).
\end{equation}

Using \eqref{theorem_2(4--5)}, \eqref{theorem_2(4--2)}, we see that
\begin{gather}\label{theorem_2(4--23)}
|I^{mh}_{53}|\le C_{12}(h,k)\sum\limits_{i=1}^n\|\partial_i
T_k(u)\chi_{\{\Omega:|u^m|\geq k\}}\|_{p_i(.)}.
\end{gather}
Due to \eqref{theorem_2(3-2)}, we get
 \begin{gather*}\partial_iT_k(u)\chi_{\{\Omega:|u^m|\geq k\}}\rightarrow \partial_iT_k(u)\chi_{\{\Omega:|u|\geq k\}}= 0\quad \mbox{a.e.\ in}\quad \Omega,\quad m\rightarrow\infty;\\
|\partial_iT_k(u)|^{p_i(\x)}\chi_{\{\Omega:|u^m|\geq k\}}\leq |\partial_iT_k(u)|^{p_i(\x)}\in L_1(\Omega), \quad m\in \mathbb{N},\quad i=1,\dots,n.
\end{gather*}
By Lemma \ref{lemma_5}, we deduce that
$$
\partial_iT_k(u)\chi_{\{\Omega:|u^m|\geq k\}}\rightarrow 0\quad \mbox{in}\quad L_{p_i(\cdot)}(\Omega),\quad  m\rightarrow\infty,\quad i=1,\dots,n.
$$
Therefore, using this fact and \eqref{theorem_2(4--23)}, we obtain
\begin{equation}\label{theorem_2(4-24)}
I^{mh}_{53}=\varepsilon_8(m).
\end{equation}

We conclude from \eqref{theorem_2(4--20)}, \eqref{theorem_2(4--21)}, \eqref{theorem_2(4--22)}, \eqref{theorem_2(4-24)} that
\begin{equation}\label{theorem_2(4-25)}
I^m_{51}\le \varepsilon_9(m)+\varepsilon(h),\quad m\geq h.
\end{equation}

Thus, using the notation \eqref{lema_7_4} ($\v^m=\nabla T_k(u^m),\;\v=\nabla T_k(u)$), we get
\begin{gather}
0\le \int\limits_{\Omega}q^m(\x) \psi_k(z^m)\eta_R(|\x|)d\x=\nonumber\\
 =I^{m}_{51}-\int\limits_{\Omega}\a(\x,T_k(u^m),\nabla T_k(u))\cdot \nabla(T_k(u^m)-T_k(u))
 \psi_k(z^m)\eta_R(|\x|)d\x\le\label{theorem_2(4-26)}\\
 \le\varepsilon_9(m)+\varepsilon(h)-\sum\limits_{i=1}^n\int\limits_{\Omega}a_i(\x,T_k(u^m),\nabla T_k(u))
\partial_i(T_k(u^m)-T_k(u))\psi_k(z^m)\eta_R(|\x|)d\x=\nonumber\\ =\varepsilon_9(m)+\varepsilon(h)
-I_{54}^m.\nonumber
\end{gather}

Thanks to \eqref{theorem_2(3-7)}, \eqref{us1'}, \eqref{theorem_2(4--3)}, we establish in the same way as in the proof of the convergence \eqref{conf_1(7)} that
\begin{gather*}
\eta_R(|\x|)\psi_k(z^m)a_i(\x,T_k(u^m),\nabla T_k (u))\rightarrow
\eta_R(|\x|)\varphi'_k(0) a_i(\x,T_k(u),\nabla T_k(u))
\,\text{ in }\, L_{{p}'_i(\cdot)}(\Omega),\quad i=1,\dots,n,
\end{gather*}
as $m\rightarrow \infty$.
Therefore, using \eqref{theorem_2(6-1)}, we deduce that
\begin{equation}\label{theorem_2(4-28)}
I^m_{54}=\sum\limits_{i=1}^n\int\limits_{\Omega}\psi_k(z^m)a_i(\x,T_k(u^m),\nabla
T_k(u))\partial_i(T_k(u^m)-T_k(u))\eta_R(|\x|)d\x=\varepsilon_{10}(m).
\end{equation}

Combining \eqref{theorem_2(4-26)}, \eqref{theorem_2(4-28)}, we get
$$
\int\limits_{\Omega}\psi_k(z^m)q^m(\x)\eta_R(|\x|) d\x\leq \varepsilon_{11}(m)+\varepsilon(h),\quad m\geq h.$$
Using \eqref{theorem_2(4--3)} and passing to the limit in the last inequality as $m\rightarrow\infty$ and then as $h\rightarrow\infty$, we obtain that
\begin{equation*}
\lim_{m\to \infty}\int\limits_{\Omega(R)}q^m(\x) d\x=0.
 \end{equation*}

By Lemma \ref{lemma_7}  and since $R>0$ is arbitrary, we have the convergence \eqref{theorem_2(4--0)}.
Using \eqref{theorem_2(4--0)}, we get for any $R, \rho>0$ the convergence
$$
\nabla T_{\rho}(u^m)\rightarrow \nabla T_{\rho}(u)\quad \mbox{a.e.\ in}\quad  \Omega(R),\quad  m\rightarrow\infty.
$$

As before, applying Egorov's theorem, we can find a set $E_{\rho}\subset\Omega(R)$ such that ${\rm meas}\,E_{\rho}<1/\rho$ and
\begin{equation}\label{theorem_2(4-29)}T_\rho (u^m) \rightarrow T_\rho (u) \; \mbox{uniformly on }\;
  \Omega(R)\setminus E_{\rho},\quad m\rightarrow\infty.
\end{equation}
Recall that
$$
{\rm meas}\,\Omega_{\rho}(R)>{\rm meas}\,\Omega(R)-1/\rho-g(\rho-2),\quad \Omega_{\rho}(R)=\{\x\in\Omega(R)\setminus E_{\rho}: |u(\x)|<\rho-1\},
$$
see \eqref{theorem_2(3--5)}.
The uniform convergence \eqref{theorem_2(4-29)} implies that $|T_\rho (u^m(\x))|<\rho$ on $\Omega_{\rho}(R)$ for any $m\geq m_0$. Therefore,
$$
\nabla u^m\rightarrow \nabla u\quad  \mbox{ a.e.\ in} \quad  \Omega_{\rho}(R),\quad m\rightarrow\infty.
$$
Thus, using the diagonalisation argument with respect to $\rho\in\mathbb{N}$, it is not hard to see that
$$
\nabla u^m\rightarrow \nabla u\quad
   \mbox{a.e.\ in}\quad \Omega(R),\quad m\rightarrow\infty.
$$

Then, by the diagonalisation argument with respect to $R\in \mathbb{N}$, we obtain the convergence along a subsequence
\begin{equation}\label{theorem_2(4-30)}
\nabla u^m\rightarrow \nabla u \quad  \mbox{ a.e.\ in} \quad \Omega,\quad m\rightarrow\infty;
  \end{equation}
\begin{equation}\label{theorem_2(4-31)}
\nabla T_k(u^m)\rightarrow \nabla T_k(u) \quad  \mbox{ a.e.\ in} \quad \Omega,\quad m\rightarrow\infty.
  \end{equation}

The continuity of $b(\x,s_0,\s)$ with respect to $(s_0,\s) $ and the convergences \eqref{theorem_2(3-2)}, \eqref{theorem_2(4-30)} yield
\begin{equation}\label{theorem_2(4-32)}
 b^m(\x,u^m,\nabla u^m)\to
b(\x,u, \nabla u)\quad\mbox{a.e.\ in} \quad \Omega,\quad m\rightarrow\infty.
\end{equation}
Using the estimate \eqref{theorem_2(2-8)} and taking into account \eqref{theorem_2(4-32)}, we obtain from the Fatou lemma that
\begin{equation}\label{theorem_2(4-33)}
b(\x,u, \nabla u)\in L_1(\Omega).
\end{equation}
Thus, the condition 1) from Definition \ref{D1} is satisfied.

\subsection{}\label{s5.6}
Let us prove that
\begin{equation}\label{theorem_2(5-1)}
b^m(\x,u^m,\nabla u^m)\to
b(\x,u,\nabla u)\quad\mbox{in} \quad L_{1,{\rm loc}}(\overline{\Omega}),\quad m\rightarrow \infty.
\end{equation}

Let $Q$ be an arbitrary subset of $\Omega$.
For any measurable set $E\subset Q$ we have
\begin{gather}\label{theorem_2(5-3)}
\int\limits_E |b^m(\x,u^m,\nabla u^m)|d\x\leq\\\leq \int\limits_{ \{E:|u^m|<h\}}|b^m(\x,u^m,\nabla u^m)|d\x+\int\limits_{\{\Omega:|u^m|\geq h\}}|b^m(\x,u^m,\nabla u^m)|d\x.\nonumber
\end{gather}
Applying \eqref{theorem_2(1-3)}, \eqref{us4}, we deduce that
$$
\int\limits_{\{E:|u^m|<h\}}|b^m(\x,u^m,\nabla u^m)|d\x\leq  \widehat{b}(h)\int\limits_{ \{E:|u^m|<h\}}\left(\P(\x,\nabla u^m)
+\Phi_0(\x)\right)d\x.$$
Due to $\Phi_0\in L_1(E)$, the convergence \eqref{theorem_2(4--0)} and the absolute continuity of integrals in the right-hand side of the last inequality, for any $\varepsilon>0$ there exists $\alpha(\varepsilon)$ such that
for every $E$ with ${\rm meas}\;E<\alpha(\varepsilon)$ the following inequalities are satisfied:
\begin{equation}\label{theorem_2(5-4)}
\int\limits_{ \{E:|u^m|<h\}}|b^m(\x,u^m,\nabla u^m)|d\x<\frac{\varepsilon}{2},\quad m\in\mathbb{N}.
\end{equation}

Combining \eqref{theorem_2(3-10)}, \eqref{theorem_2(5-3)}, \eqref{theorem_2(5-4)}, we obtain that
$$
\int\limits_E |b^m(\x,u^m,\nabla u^m)|d\x< \varepsilon\quad \forall\;E\quad \mbox{such that}\; {\rm meas}\;E<\alpha(\varepsilon),\quad m\in\mathbb{N}.
$$
This implies that the sequence $\{b^m(\x,u^m,\nabla u^m)\}_{m\in \mathbb{N}}$ has uniformly absolutely continuous integrals over the set $Q$. By Lemma \ref{lemma_9}, there is the convergence
$$
b^m(\x,u^m,\nabla u^m)\to
b(\x,u,\nabla u)\quad\mbox{in} \quad L_{1}(Q),\quad m\rightarrow \infty,
$$
for any bounded set $Q\subset \Omega$. The convergence \eqref{theorem_2(5-1)} is proved.

\subsection{}\label{s5.7}

To prove $\eqref{intn}$, we take $v=T_k(u^m-\xi)$ with $\xi\in C_0^1(\Omega)$ as a test function in \eqref{theorem_2(1-5)}, and get
\begin{gather}\label{theorem_2(7-1)}
\int\limits_{\Omega} \a(\x,T_m(u^m),\nabla u^m)\cdot\nabla T_k(u^m-\xi)d\x+\\
+\int\limits_{\Omega} \left(b^m(\x,u^m,\nabla u^m)+|u^m|^{p_0(\x)-2}u^m+f^m\right)T_k(u^m-\xi)d\x=I^m+J^m=0.\nonumber
\end{gather}

Let $M=k+\|\xi\|_{\infty}$. If $|u^m|\geq M$, then $|u^m-\xi|\geq |u^m|-\|\xi\|_{\infty}\geq k$. Therefore, $\{\Omega:|u^m-\xi|< k\}\subseteq \{\Omega:|u^m|< M\}$, and hence
\begin{gather}
I^m=\int\limits_{\Omega} \a(\x,T_m(u^m),\nabla u^m)\cdot\nabla T_k(u^m-\xi)d\x=\nonumber\\
=\int\limits_{\Omega} \a(\x,T_M(u^m),\nabla T_M(u^m))\cdot\nabla T_k(u^m-\xi)d\x=\label{theorem_2(7-2)}\\=\int\limits_{\Omega} \a(\x,T_M(u^m),\nabla T_M(u^m))\cdot(\nabla T_M(u^m)-\nabla\xi)\chi_{\{\Omega:|u^m-\xi|< k\}}d\x,\quad m\geq M.\nonumber
\end{gather}

Let $v^m=u^m-\xi$, $v=u-\xi$. Since $|v|= k$ in the set where $|v^m|\to k$ as $m\rightarrow\infty$, we have $\nabla v=0$  a.a.\ in $\Omega$. Consequently,
\begin{gather}
\nabla T_k(v^m)-\nabla T_k(v)=\chi_{\{\Omega:|v^m|<k\}}(\nabla v^m-\nabla
v)+\nonumber\\+\left(\chi_{\{\Omega:|v^m|<k\}}-\chi_{\{\Omega:|v|<k\}}\right)\nabla v\to 0\quad \mbox{a.e.\ in} \quad \Omega,\quad m\rightarrow\infty.\label{theorem_2(7-4)}
\end{gather}
Using the inequality \eqref{ung} and the assumptions \eqref{us1'}, \eqref{us3}, we deduce for any $\varepsilon\in (0,1)$ that
\begin{gather*}
a(\x,T_M(u^m),\nabla T_M(u^m))\cdot(\nabla
T_M(u^m)-\nabla\xi)\chi_{\{\Omega:|u^m-\xi|<
k\}}\ge\\\ge\left((\overline{a}-\varepsilon
\widehat{A}^n(M))\P(\x,\nabla
T_M(u^m))-\widehat{A}_+(M)\Psi^n(\x)-\phi(\x)-C_{13}(\varepsilon)\P(\x,\nabla
\xi)\right)\chi_{\{\Omega:|u^m-\xi|<
k\}}.
\end{gather*}
Taking $\varepsilon<\overline{a}/\widehat{A}^n(M)$, we get the inequality
\begin{gather*}
a(\x,T_M(u^m),\nabla T_M(u^m))\cdot\nabla( T_M(u^m)-\xi)\chi_{\{\Omega:|u^m-\xi|<
k\}}\ge\\\ge -C_{14}(\Psi^n(\x)+\P(\x,\nabla \xi))-\phi(\x)\in L_1(\Omega).
\end{gather*}
From the convergences \eqref{theorem_2(7-4)}, \eqref{theorem_2(3-7)}, \eqref{theorem_2(4-31)}, the continuity of $\a(\x,s_0,\s)$ with respect to $(s_0,\s)$, and Fatou's lemma, we have
\begin{gather}\lim_{m\to \infty}\inf I^m\geq \int\limits_{\Omega}\a(\x,T_M(u),\nabla T_M(u))\cdot \nabla T_k(
u-\xi)d\x=\label{theorem_2(7-5)}\\
=\int\limits_{\Omega}\a(\x,u,\nabla u)\cdot \nabla T_k(
u-\xi)d\x.\nonumber
\end{gather}

Using Lemma \ref{lemma_8}, we get from \eqref{theorem_2(3-2)} that
\begin{equation}\label{theorem_2(7-6)}
 T_k(u^m-\xi)\stackrel {*}\rightharpoonup T_k(u-\xi)\quad \text{in} \quad L_{\infty}(\Omega),\quad m\rightarrow \infty.
\end{equation}

Let us split $J^m$ into two summands. The first integral
$$
J^m_1=\int\limits_{\Omega} \left(b^m(\x,u^m,\nabla u^m)+|u^m|^{p_0(\x)-2}u^m\right)T_k(u^m-\xi)d\x
$$
can be estimated as follows. Let ${\rm supp}\,\xi\subset \Omega(l),\;l\geq l_0,\;$   $c^m(\x,u^m,\nabla u^m)=b^m(\x,u^m,\nabla u^m)+|u^m|^{p_0(\x)-2}u^m$, $c(\x,u,\nabla u)=b(\x,u,\nabla u)+|u|^{p_0(\x)-2}u$. Then, recalling  \eqref{theorem_2(1-4)}, we get for $l\geq l_0$ that
\begin{gather*}
J^m_1=\int\limits_{\Omega\backslash \Omega(l)}c^m(\x,u^m,\nabla u^m)T_k(u^m)d\x+\int\limits_{\Omega(l)}c^m(\x,u^m,\nabla u^m)T_k(v^m)d\x\geq\\\geq\int\limits_{\Omega(l)}c^m(\x,u^m,\nabla u^m)T_k(v^m)d\x= \overline{J}^{lm}_1.
\end{gather*}
Applying \eqref{theorem_2(3-8)}, \eqref{theorem_2(5-1)},  \eqref{theorem_2(7-6)}, we pass to the limit as $m\to \infty$ and then as $l\to \infty$ to obtain
\begin{equation}
\label{theorem_2(7-7)}\int\limits_{\Omega}(b(\x,u,\nabla u)+|u|^{p_0(\x)-2}u)T_k(u-\xi)d\x=\lim_{l\to \infty}\lim_{m\to \infty}\overline{J}^{lm}_1\leq \lim_{m\to \infty}\inf J_1^m.
\end{equation}

Using \eqref{theorem_2(1-0)}, \eqref{theorem_2(7-6)}, and passing to the limit as $m\rightarrow\infty$ in the second integral, we conclude that
\begin{equation}\label{theorem_2(7-8)}
J^m_2=\int\limits_{\Omega}f^m T_k(u^m-\xi)d\x\rightarrow\int\limits_{\Omega}f T_k(u-\xi)d\x.
\end{equation}

Combining \eqref{theorem_2(7-1)}, \eqref{theorem_2(7-5)}, \eqref{theorem_2(7-7)}, \eqref{theorem_2(7-8)}, we derive \eqref{intn}.
\qed

\section{Renormalized solution}\label{s6}
In this section, we prove that the entropy solution is a renormalized solution of the problem \eqref{ur}, \eqref{gu}, \eqref{us6}.

\begin{definition}\label{D3}
	A renormalized solution of the problem \eqref{ur}, \eqref{gu}, \eqref{us6} is a function $u\in \mathring{\mathcal{T}}^1_{\overrightarrow{\bf p}(\cdot)}(\Omega)$ such that
	\begin{enumerate}
		\item[1)] $B(\x)=b(\x,u,\nabla u)\in L_{1}(\Omega)$;
		\item[2)] $\lim\limits_{h\rightarrow \infty}\int\limits_{h\leq |u|<  h+1}\P(\x,\nabla u) d\x=0$;
		\item[3)] for any smooth function $S\in W^1_{\infty}(\mathbb{R})$ with compact support and any function $\xi\in C_0^1(\Omega)$ there holds
		\begin{equation}
		\label{ren}
		\langle(b(\x,u, \nabla u) +|u|^{p_0(\x)-2}u+f)S(u)\xi\rangle
		+\langle \a(\x,u,\nabla u)\cdot(S'(u)\xi\nabla u+S(u)\nabla\xi)\rangle= 0.
		\end{equation}
	\end{enumerate}
\end{definition}

\begin{theorem}\label{t3}
	Let the assumptions \eqref{us0}--\eqref{us5} be satisfied.
	Then the entropy solution $u\in \mathring {\mathcal{T}}^1_{\overrightarrow{\bf p}(\cdot)}(\Omega)$ obtained in Theorem $\ref{t1}$ is a renormalized solution of the problem \eqref{ur}, \eqref{gu}, \eqref{us6}.
\end{theorem}

\begin{proof}
We prove that the entropy solution satisfies Definition \ref{D3} of a renormalized solution.
The condition 1) holds true since it coincides with the condition 1) of Definition \ref{D1}.
The condition 2) is also satisfied, see \eqref{lemma_3_}.

Let us prove the equality \eqref{ren}.
Let $\{u^m\}_{m\in \mathbb{N}}$ be a sequence of weak solutions of the problem \eqref{theorem_2(1-2)}, \eqref{gu}.
Let $S\in W^1_{\infty}(\mathbb{R})$ be such that ${\rm supp}\;S\subset[-M,M]$ for $M>0$. For any function $\xi\in C_0^1(\Omega)$, taking $S(u^m)\xi\in\mathring
{W}_{\overrightarrow{\bf p}(\cdot)}^1(\Omega)$ as a test function in \eqref{theorem_2(1-5)}, we derive that
 \begin{gather}
\langle \a(\x,T_m(u^m), \nabla u^m)\cdot(S'(u^m)\xi\nabla u^m+S(u^m)\nabla\xi)\rangle+\nonumber\\+\langle (f^m(\x)+b^m(\x,u^m,\nabla u^m)+|u^m|^{p_0(\x)-2}u^m)S(u^m)\xi\rangle=I^m+J^m=0,\quad m\in \mathbb{N}.\label{7-1}
\end{gather}

Clearly,
\begin{gather*}
I^m=\int\limits_{\Omega}\a(\x,T_m (u^m), \nabla u^m)\cdot(S'(u^m)\xi\nabla u^m+S(u^m)\nabla\xi)d\x=\\
=\int\limits_{\Omega}\a(\x,T_M(u^m), \nabla T_M(u^m))\cdot(S'(u^m)\xi\nabla T_M (u^m)+S(u^m)\nabla\xi)d\x,\quad m\geq M.\end{gather*}

As in Lemma \ref{lemma_7} (see \eqref{conf_1(6)}), we obtain the weak convergence
 \begin{gather}\label{7-2}
\a(\x,T_M(u^m),\nabla T_M(u^m))\rightharpoonup \a(\x,T_M(u),\nabla T_M(u))\quad\mbox{in}\quad \L_{\overrightarrow{\p}'(\cdot)}(\Omega),\quad m\rightarrow \infty.
\end{gather}
Using the convergences \eqref{theorem_2(3-2)},  \eqref{theorem_2(4--0)}, and noting that ${\rm supp}\,\xi$ is a bounded subset of $\Omega$, we deduce from Lemma \ref{lemma_8} that
\begin{gather}\label{7-3}
S'(u^m)\xi\nabla T_M(u^m)+S(u^m)\nabla\xi\rightarrow S'(u)\xi\nabla T_M(u)+S(u)\nabla\xi\quad\mbox{in}\quad \L_{\overrightarrow{\p}(\cdot)}(\Omega),
\end{gather}
as $m\rightarrow \infty$.
Combining \eqref{7-2}, \eqref{7-3}, we get
\begin{gather}
\lim_{m\to \infty} I^m=\int\limits_{\Omega} \a(\x,T_M(u),\nabla T_M(u))\cdot(S'(u)\xi\nabla T_M (u)+S(u)\nabla\xi)d\x=\nonumber
\\=\int\limits_{\Omega}\a(\x,u,\nabla u)\cdot(S'(u)\xi\nabla u+S(u)\nabla\xi)d\x.\label{7-4}
\end{gather}

By Lemma \ref{lemma_8} we have
$$
 S(u^m)\xi\stackrel{*}\rightharpoonup S(u)\xi\quad \text{in} \quad L_{\infty}(\Omega),\quad m\rightarrow \infty.
$$
Therefore, in view of \eqref{theorem_2(1-0)}, \eqref{theorem_2(3-8)}, \eqref{theorem_2(5-1)}, we obtain
\begin{gather}\label{7-5}
\lim_{m\to \infty} J^m=\int\limits_{\Omega}(f+b(\x,u,\nabla u)+|u|^{p_0(\x)-2}u)S(u)\xi d\x.
\end{gather}
 Combining \eqref{7-1}, \eqref{7-4}, \eqref{7-5}, we get the equality \eqref{ren}.
 Thus, we come to the conclusion that $u$ is a renormalized solution of the problem \eqref{ur}, \eqref{gu}.
\end{proof}

\section{Appendix}\label{s7}

\begin{theorem}\label{t2}
Let the assumption \eqref{us0}--\eqref{us5} be satisfied. Then there exists a generalized solution of the problem \eqref{theorem_2(1-2)}, \eqref{gu}.
\end{theorem}

The proof of this theorem is based on the assertion that the operator $A$ is pseudo-monotone.

\begin{definition}\label{kozhelm-def_3}
Let $V$ be a reflexive Banach space. An operator $A:V\rightarrow V'$ is pseudo-monotone if
\begin{enumerate}
\item[(i)] $A$ is a bounded operator;
\item[(ii)] assumptions
\begin{equation}
\begin{array}{l}
u^j\rightharpoonup u\quad \mbox{in}\quad V,\\
A(u^j)\rightharpoonup \omega \quad\mbox{in}\quad V',\\
 \lim\limits_{j\to \infty}\sup <A(u^j),u^j>\,\leq\, <\omega ,u>\label{psevdomon1}
 \end{array}
 \end{equation} imply that
\begin{equation}\label{psevdomon2}
\omega=A(u),\quad \lim_{j\to \infty}\inf< A(u^j),u^j>=< A(u),u>.
\end{equation}
\end{enumerate}
\end{definition}

\begin{lemma}[\protect{\cite[Chapter II, \S2, Theorem 2.7]{Lions}}]\label{kozhelm-lemma_3}
Let V be a reflexive separable Banach space.
Assume that an operator $A:V\rightarrow V'$ has the following properties:
$A$ is pseudo-monotone and coercive, i.e.,
\begin{equation}\label{kozhelm-koerc}
\frac{< A(u),u>}{\|u\|}\rightarrow\infty,\quad \|u\|\rightarrow\infty.
\end{equation}
Then the map $A:V\rightarrow V'$ is surjective, i.e., for any $F\in V'$ there exists $u\in V$ such that $A(u)=F$.
\end{lemma}

\begin{propos}\label{kozhelm-confir_1}
	Let the assumptions \eqref{us0}--\eqref{us5} be satisfied. Then the operator
 $$
 {\bf A}^m=\A^m+{A}^m_0:\mathring{W}_{\overrightarrow{\bf p}(\cdot)}^{1}(\Omega)\rightarrow \left(\mathring{W}_{\overrightarrow{\bf p}(\cdot)}^{1}(\Omega)\right)'
 $$
 defined by \eqref{theorem_2(1-6)} is pseudo-monotone and coercive.
\end{propos}
\begin{proof}
The inequalities \eqref{us1'}, \eqref{st2}, imply the estimates
$$
\rho_{p_i'(\cdot)}(a_i(\x,T_m(u),\nabla u))\leq \widehat{A}_i(m)(\|\P(\x,\nabla u)\|_1+\|\Psi_i\|_1)\leq C_{i1}(m,\|\nabla u\|_{\overrightarrow{\p}(\cdot)})
$$
for $i=1,\ldots,n$.
Then, using \eqref{st3}, we get
\begin{equation}\label{theorem_3(4)}
\|\a(\x,T_m(u),\nabla u)\|_{\overrightarrow{\p}'(\cdot)}=\sum\limits_{i=1}^n\|a_i(\x,T_m(u),\nabla u)\|_{p'_i(\cdot)}\leq C_2(m,\|\nabla u\|_{\overrightarrow{\p}(\cdot)}).
\end{equation}

The inequality \eqref{theorem_2(1-3)} yields
\begin{equation}\label{theorem_3(5)}
\|b^m(\x,u,\nabla u)\|_{p_0'(\cdot)}\leq C_3(m).
\end{equation}

Applying \eqref{st4}, we deduce that
\begin{equation}\label{theorem_3(6)}
\||u|^{p_0(\x)-2}u\|_{p_0'(\cdot)}\leq C_4(\|u\|_{p_0(\cdot)}).
\end{equation}

Using the estimates \eqref{theorem_3(4)}--\eqref{theorem_3(6)} and the inequality \eqref{gel}, for any $v\in \mathring{W}^{1}_{\overrightarrow{\bf p}(\cdot)}(\Omega)$ we obtain that
\begin{gather}
|<\A^m(u),v>|\leq \sum\limits_{i=1}^n\int\limits_{\Omega}\left|a_i(\x,T_m(u),\nabla u)\right||v_{x_i}|d\x\leq\nonumber\\
\leq 2\sum\limits_{i=1}^n\|a_i(\x,T_m(u),\nabla u\|_{p'_i(\cdot)}\|v_{x_i}\|_{p_i(\cdot)}\leq\nonumber\\
\leq 2\|\a(\x,T_m(u),\nabla u)\|_{\overrightarrow{\p}'(\cdot)}\|\nabla v\|_{\overrightarrow{\p}(\cdot)}\leq C_{5}(m,\|\nabla u\|_{\overrightarrow{\p}(\cdot)})\|\nabla v\|_{\overrightarrow{\p}(\cdot)},\label{theorem_3(7)}
\end{gather}
\begin{gather}
|<A^m_0(u),v>|\leq \int\limits_{\Omega}\left|b^m(\x,u,\nabla u)+|u|^{p_0(\x)-2}u+f^m\right||v|d\x\leq\nonumber\\
\leq 2\left(\|b^m(\x,u,\nabla u)\|_{p'_0(\cdot)}+\||u|^{p_0(\x)-2}u\|_{p'_0(\cdot)}+\|f^m\|_{p'_0(\cdot)}\right)\|v\|_{p_0(\cdot)}\nonumber \\
\leq C_{6}(m,\| u\|_{p_0(\cdot)})\| v\|_{p_0(\cdot)}.\label{theorem_3(8)}
\end{gather}
The boundedness of the operator ${\bf A}^m$ follows from \eqref{theorem_3(7)}, \eqref{theorem_3(8)}.

Let us prove the coercivity of ${\bf A}^m$. Using \eqref{us3}, \eqref{theorem_2(1-4)}, we get
\begin{gather*}
<\mathbf{A}^m(u),u>=\int\limits_\Omega \a(\x,T_m(u),\nabla u)\cdot\nabla u d{\x}+\nonumber\\
+\int\limits_\Omega (b^m(\x,u,\nabla u)u + |u|^{p_0(\x)}+f^mu)d\x\geq\nonumber\\
\geq\overline{a}\sum\limits_{i=1}^n\rho_{p_i(\cdot)}(u_{x_i})+(1-\varepsilon)\rho_{p_0(\cdot)}(u)
-C_{\varepsilon}\rho_{p'_0(\cdot)}(f^m)-\|\phi\|_1.\nonumber
\end{gather*}
Then, applying \eqref{st2} and taking $\varepsilon\in(0,1)$, we deduce the inequality
\begin{gather}\label{theorem_3(9)}<\mathbf{A}^m(u),u>\geq
 C_{7}\sum_{i=1}^n
\|u_{x_i}\|_{p_i(\cdot)}^{p^-_i}+C_7\|u\|_{p_0(\cdot)}^{p^-_0}-C_8.
\end{gather}

Let $\|u^j\|_{\mathring{W}^{1}_{\overrightarrow{\bf p}(\cdot)}(\Omega)}\rightarrow \infty$ as $j\rightarrow \infty$.
Then for any $l>1$ there exists $j_0$ such that for each $j\geq j_0$ the following inequality is satisfied:
\begin{equation}\label{theorem_3(10)}
\|u^j\|_{\mathring{W}^{1}_{\overrightarrow{\bf p}(\cdot)}(\Omega)}=\|u^j\|_{p_0(\cdot)}+\sum\limits_{i=1}^n\|u^j_{x_i}\|_{p_i(\cdot)}>l(n+1).
\end{equation}
For any $j\geq j_0$ there exists at least one summand which is greater than $l$.
Assume, for definiteness, that for a fixed $j\geq j_0$ the summand $\|u^j\|_{p_0(\cdot)}>l$ is largest.
Combining \eqref{theorem_3(9)}, \eqref{theorem_3(10)}, we have
$$
\frac{<\mathbf{A}^m(u^j),u^j>}{\|u^j\|_{\mathring{W}^{1}_{\overrightarrow{\bf p}(\cdot)}(\Omega)}}\geq C_{9}\frac{\|u^j\|^{p^-_0-1}_{p_0(\cdot)}}{(n+1)}-\frac{C_{10}}{(n+1)l}>C_{11}l^{p^-_0-1}-\frac{C_{12}}{l}.
$$
Therefore, in view of the arbitrariness of $l$ and $j,\;$ $j\geq j_0$, we get \eqref{kozhelm-koerc}.

Let us now prove that the assumptions \eqref{psevdomon1} imply \eqref{psevdomon2}. Namely, we show that if
\begin{gather}
u^j\rightharpoonup u\quad \mbox{in}\quad \mathring{W}^{1}_{\overrightarrow{\bf p}(\cdot)}(\Omega),\label{theorem_3(11)}\\
{\bf A}^m(u^j)\rightharpoonup \omega \quad\mbox{in}\quad {W}^{-1}_{\overrightarrow{\bf p}'(\cdot)}(\Omega),\label{theorem_3(12)}\\
 \lim\limits_{j\to \infty}\sup<{\bf A}^m(u^j),u^j>\leq <\omega ,u>,\label{theorem_3(13)}
\end{gather}
then
\begin{gather}\label{theorem_3(14)}
\omega={\bf A}^m(u),\\
 \lim_{j\to \infty}<{\bf A}^m(u^j),u^j>= < {\bf  A}^m(u),u>.\label{theorem_3(15)}
\end{gather}

The convergence \eqref{theorem_3(11)} yields the estimate
\begin{equation}\label{theorem_3(16)}
\|u^j\|_{\mathring{W}^{1}_{\overrightarrow{\bf p}(\cdot)}(\Omega)}\leq C_{13},
\quad j\in\mathbb{N}.
\end{equation}
Let us show the following convergence along a subsequence:
\begin{equation}\label{theorem_3(17)}
u^j\rightarrow u  \quad \mbox {a.e.\ in} \quad \Omega,\quad j\rightarrow \infty.
\end{equation}

For $R>0$, applying \eqref{sum}, we deduce that
\begin{gather*}
\int\limits_{\Omega}\P(\x,\nabla (\eta_{R}(|\x|) u^j))d\x\leq C_{14} \int\limits_{\Omega }\P(\x,\eta_{R}\nabla u^j ) d\x
+C_{14}\int\limits_{\Omega}\P(\x,\nabla \eta_{R}u^j)d\x\leq\\
\leq C_{14}\sum\limits_{i=1}^n\left(\rho_{p_i(\cdot)}(u^j_{x_i})+\rho_{p_i(\cdot),\Omega(R)}(u^j)\right).
\end{gather*}
Therefore, using the inequality \eqref{us0} and the estimate \eqref{theorem_3(16)}, we get
\begin{gather*}\int\limits_{\Omega}\P(\x,\nabla (\eta_{R}(|\x|) u^j))d\x\leq C_{14}\|{\bf P}(\x,u^j,\nabla u^j)\|_1+ C_{14}{\rm meas}\, \Omega (R)\leq C_{15},\quad j\in\mathbb{N}.
\end{gather*}
Thus, for any fixed $R>0$ we have the boundedness of the set $\{\eta_{R} u^j\}_{j\in\mathbb{N}}$ in
$\mathring{H}_{\overrightarrow{\p}(\cdot)}^1(\Omega(R+1))$.
By Lemma \ref{lemma_0}, the space $\mathring{H}_{\overrightarrow{\bf p}(\cdot)}^{1}(\Omega(R+1))$ is compactly embedded in $L_{p_-(\cdot)}({\Omega}(R+1))$. Hence, for any fixed $R>0$ we have the convergence of $\eta_{R}u^j$ in $L_{p_-(\cdot)}(\Omega(R+1))$ as $j\rightarrow\infty$. This implies the convergence
 \begin{equation}\label{theorem_3(178)}
 u^j\rightarrow u\quad \text{in} \quad L_{p_-(\cdot)}(\Omega(R))
 \end{equation}
 and the convergence $u^j\rightarrow u$ a.e.\ in $\Omega(R)$ along a subsequence.
 Then, using the diagonalisation argument with respect to $R\in \mathbb{N}$, we get the convergence \eqref{theorem_3(17)}.

From \eqref{theorem_3(16)}, \eqref{theorem_3(4)} we have the estimate
\begin{gather}
\|\a(\x,T_m(u^j),\nabla u^j)\|_{\overrightarrow{\p}'(\cdot)}\leq C_{16}(m),
\quad j\in\mathbb{N}.\label{theorem_3(18)}
\end{gather}
Therefore, there exist functions $\widetilde{\a}^m\in \L_{\overrightarrow{\p}'(\cdot)}(\Omega)$ such that
\begin{equation}\label{theorem_3(19)}
\a(\x,T_m(u^j),\nabla u^j)\rightharpoonup \widetilde{\a}^m\quad \text{in} \quad \L_{\overrightarrow{\p}'(\cdot)}(\Omega),\quad j\rightarrow \infty.
\end{equation}

The estimate \eqref{theorem_3(5)} implies the existence of a function $\widetilde{b}^m\in L_{p'_0(\cdot)}(\Omega)$ such that
\begin{equation}\label{theorem_3(21)}
b^m(\x,u^j,\nabla u^j)\rightharpoonup \widetilde{b}^m\quad \text{in} \quad L_{p'_0(\cdot)}(\Omega),\quad j\rightarrow \infty.
\end{equation}

Then, the estimates \eqref{theorem_3(6)}, \eqref{theorem_3(16)} yield
$$
\||u^j|^{p_0(\x)-2}u^j\|_{p_0'(\cdot)}\leq C_{17},\quad j\in\mathbb{N}.
$$
Thus, in view of the convergence \eqref{theorem_3(17)}, we obtain from Lemma \ref{lemma_4} that
\begin{equation}\label{theorem_3(21-2)}
|u^j|^{p_0(\x)-2}u^j\rightharpoonup|u|^{p_0(\x)-2}u\quad \text{in} \quad L_{p'_0(\cdot)}(\Omega),\quad j\rightarrow \infty.
\end{equation}

Due to \eqref{theorem_3(12)}, using \eqref{theorem_3(19)}--\eqref{theorem_3(21)}, for any $v\in \mathring{W}_{\overrightarrow{\bf p}(\cdot)}^{1}(\Omega)$ we deduce that
\begin{gather}
<\omega,v>=\lim_{j\to \infty} < {\bf A}^m(u^j),v>=
\lim_{j\to \infty}\langle\a(\x,T_m(u^j),\nabla u^j)\cdot\nabla v\rangle +\nonumber\\
+\lim_{j\to \infty}\langle(b^m(\x,u^j,\nabla u^j)+|u^j|^{p_0(\x)-2}u^j+f^m)v\rangle=\label{theorem_3(22)}\\
=\langle\widetilde{\a}^m\cdot\nabla v\rangle +\langle(\widetilde{b}^m+|u|^{p_0(\x)-2}u+f^m)v\rangle.\nonumber
\end{gather}
Evidently, the following equality is satisfied:
\begin{gather}
 < {\bf A}^m(u^j),u^j>=
\langle\a(\x,T_m(u^j),\nabla u^j)\cdot\nabla u^j\rangle +\nonumber\\
+\langle(b^m(\x,u^j,\nabla u^j)+|u^j|^{p_0(\x)-2}u^j+f^m)u^j\rangle.\label{theorem_3(22-3)}
\end{gather}

On one hand, relations \eqref{theorem_3(13)}, \eqref{theorem_3(22)} yield
\begin{gather}
\lim\limits_{j\to \infty}\sup <{\bf A}^m(u^j),u^j>
\leq\langle\widetilde{\a}^m\cdot\nabla u\rangle +\langle(\widetilde{b}^m+|u|^{p_0(\x)-2}u+f^m)u\rangle.\label{theorem_3(23)}
\end{gather}
It follows from the convergence \eqref{theorem_3(11)} that
\begin{gather}
\lim\limits_{j\to \infty}\langle f^m u^j\rangle=\langle f^m u\rangle.\label{theorem_3(25)}
\end{gather}
Then, from the inequality \eqref{theorem_2(1-3)} and the convergence \eqref{theorem_3(178)} we have
\begin{gather*}
\lim\limits_{j\to \infty}\left|\langle b^m(\x,u^j,\nabla u^j)(u^j-u)\rangle\right|\leq m\lim\limits_{j\to \infty}\int\limits_{\Omega(m)}|u^j-u|d\x
\leq \\ \leq C(m)\lim\limits_{j\to \infty}\|u^j-u\|_{p_-(\cdot),\Omega(m)}= 0.
\end{gather*}
Using this fact and the convergence \eqref{theorem_3(21)}, we conclude that
\begin{equation}\label{theorem_3(26)}
\lim\limits_{j\to \infty}\langle b^m(\x,u^j,\nabla u^j)u^j\rangle= \langle \widetilde{b}^m u\rangle.
\end{equation}
Combining \eqref{theorem_3(23)}--\eqref{theorem_3(26)}, we obtain the inequality
\begin{gather}\label{theorem_3(27)}
\lim\limits_{j\to \infty}\sup\langle \a(\x,T_m(u^j),\nabla u^j)\cdot\nabla u^j+|u^j|^{p_0(\x)-2}u^j\rangle
\leq\langle\widetilde{\a}^m\cdot\nabla u+|u|^{p_0(\x)}\rangle.
\end{gather}

On the other hand, thanks to the assumption \eqref{us2}, we have
\begin{gather*}
 \langle (\a(\x,T_m(u^j),\nabla u^j)-\a(\x,T_m(u^j),\nabla u))\cdot\nabla (u^j-u)\rangle+\\+\langle(|u^j|^{p_0(\x)-2}u^j-|u|^{p_0(\x)-2}u)(u^j-u)\rangle\geq0.
 \end{gather*}
Then
 \begin{gather}\label{theorem_3(28)}
 \langle \a(\x,T_m(u^j),\nabla u^j)\cdot\nabla u^j+|u^j|^{p_0(\x)}\rangle\geq\langle\a(\x,T_m(u^j),\nabla u^j)\cdot\nabla u\rangle+\\+\langle\a(\x,T_m(u^j),\nabla u)\cdot\nabla (u^j-u)\rangle+\langle|u^j|^{p_0(\x)-2}u^ju\rangle+\langle|u|^{p_0(\x)-2}u(u^j-u)\rangle.\nonumber
 \end{gather}

Using the convergence \eqref{theorem_3(17)}, the inequality \eqref{us1'} and Lemma \ref{lemma_5}, we establish that
\begin{equation}\label{theorem_3(29)}
\a(\x,T_m(u^j),\nabla u)\rightarrow \a(\x,T_m(u),\nabla u)
\,\text{  strongly in }\, \L_{\overrightarrow{\p}'(\cdot)}(\Omega),\quad j\rightarrow\infty.
\end{equation}

In view of the convergences \eqref{theorem_3(11)},  \eqref{theorem_3(19)}, \eqref{theorem_3(21-2)}, \eqref{theorem_3(21)}, \eqref{theorem_3(29)}, we deduce from \eqref{theorem_3(28)} that
 \begin{gather}\label{theorem_3(30)}
\lim\limits_{j\to \infty}\inf\langle \a(\x,T_m(u^j),\nabla u^j)\cdot\nabla u^j+|u^j|^{p_0(\x)-2}u^j\rangle\geq\langle\widetilde{\a}^m\cdot\nabla u+|u|^{p_0(\x)}\rangle.
\end{gather}
Combining \eqref{theorem_3(27)}, \eqref{theorem_3(30)}, we get
\begin{gather}\label{theorem_3(31)}
\lim\limits_{j\to \infty}\langle \a(\x,T_m(u^j),\nabla u^j)\cdot\nabla u^j+|u^j|^{p_0(\x)-2}u^j\rangle=\langle\widetilde{\a}^m\cdot\nabla u+|u|^{p_0(\x)}\rangle.
\end{gather}

Now, we obtain from  \eqref{theorem_3(22-3)}, \eqref{theorem_3(25)}, \eqref{theorem_3(26)}, \eqref{theorem_3(31)}, \eqref{theorem_3(22)} that
\begin{equation}\label{theorem_3(32)}
\lim_{j\to \infty}\langle {\bf A}^m(u^j),u^j\rangle= < \omega,u>.
\end{equation}
Combining \eqref{theorem_3(11)},  \eqref{theorem_3(19)}, \eqref{theorem_3(21)}, \eqref{theorem_3(29)}, \eqref{theorem_3(31)}, we get the relation
\begin{gather*}\lim\limits_{j\to \infty}\langle (\a(\x,T_m(u^j),\nabla u^j)-\a(\x,T_m(u^j),\nabla u))\cdot\nabla (u^j-u)+\\+\langle(|u^j|^{p_0(\x)-2}-|u|^{p_0(\x)-2}u)(u^j-u)\rangle= 0.
\end{gather*}
Therefore, Lemma \ref{lemma_7} ($\v^j=\nabla u^j,\;\v=\nabla u$) implies the following convergences along a subsequence:
\begin{gather*}
\nabla u^j\rightarrow \nabla u \quad \mbox{in}\quad  \L_{\overrightarrow{\p}(\cdot)}(\Omega),\quad j\rightarrow \infty,\\
\nabla u^j\rightarrow \nabla u \quad \mbox{a.e.\ in}\quad \Omega,\quad j\rightarrow \infty.
\end{gather*}
Then Lemma \ref{lemma_4}
$$
\widetilde{\a}^m=\a(\x,T_m(u),\nabla u),\quad \widetilde{b}^m=b(\x,u,\nabla u)
$$
and, in view of \eqref{theorem_3(22)}, for any $v\in\mathring{W}_{\overrightarrow{\bf p}(\cdot)}^{1}(\Omega)$ the following equality is satisfied:
$$
< \omega,v>=< {\bf A}^m(u),v>.
$$
Finally, thanks to \eqref{theorem_3(32)}, we obtain \eqref{theorem_3(14)} and \eqref{theorem_3(15)}.

From Proposition \ref{kozhelm-confir_1} and in view of Lemma \ref{kozhelm-lemma_3}, there exists a function $u\in \mathring{W}_{\overrightarrow{\bf p}(\cdot)}^{1}(\Omega)$ such that ${\bf A}^m(u)={\bf O}.$
Thus, for any $v\in\mathring{W}_{\overrightarrow{\bf p}(\cdot)}^{1}(\Omega)$ the integral identity \eqref{theorem_2(1-5-6)} holds true.
\end{proof}

\bigskip
\noindent
{\bf Acknowledgements.} This work was supported by the RFBR under Grant [18-01-00428].


\selectlanguage{english}


\begin{thebibliography}{99}

\RBibitem{Beni}
\by Ph.~Benilan,  L.~Boccardo~,  Th.~Gallou\"{e}t,  R.~Gariepy, M.~Pierre~, J.L.~Vazquez
\paper An $L^1$-theory of existence and nuniqueness of solutions of nonlinear elliptic equations
\jour Annali della Scuola Normale Superiore di Pisa, Classe di Scienze
\vol 22
\issue 2
\pages 241--273
\yr 1995

\RBibitem{Kru}
\by S.~N.\ Kruzhkov
\paper First order quasilinear equations in several independent variables
\jour Mathematics of the USSR-Sbornik
\vol 10
\issue 2
\pages 217-–243
\yr 1970

\RBibitem{LionDi}
\by R.\,J.~DiPerna and P.-L.~Lions
\paper On the
Cauchy problem for Boltzmann equations: global existence and weak
stability
\jour Ann. of Math.
\vol 130
\pages 321-366
\yr  1989

\RBibitem{Boccardo_1996}
\by L. Boccardo
\paper Some nonlinear Dirichlet problems in L1 involving lower order terms in divergence form
\jour Pitman Res. Notes Math. Ser.
\vol 350
\pages 43--57
\yr 1996


\RBibitem{Boccardo_1993}
\by L.~Boccardo, D.~Giachetti, J.I.~Diaz, F.~Murat
\paper Existence and regularity of renormalized
solutions for some elliptic problems involving derivatives of nonlinear terms
\jour J. Differential
Equations
\vol 106
\issue 2
\pages 215--237
\yr 1993

\RBibitem{Koval_UMJ}
\by A.A.~Kovalevskii
\paper On the convergence of functions from a Sobolev space satisfying special integral estimates
\jour Ukr. Math. J.
\vol 58
\issue 2
\pages 189--205
\yr 2006


\RBibitem{BB}
\by A.~Benkirane, J.~Bennouna
\paper Existence of entropy solutions for some elliptic problems involving derivatives of nonlinear terms in Orlicz spaces
\jour Abstr. Appl. Anal.
\vol 7
\issue 2
\pages 85--102
\yr 2002

\RBibitem{ABT}
\by L.~Aharouch, J.~Bennouna, A.~Touzani
\paper Existence of renormalized solution of some elliptic problems in Orlicz spaces
\jour Rev. Mat. Complut.
\vol 22
\issue 1
\pages 91--110
\yr 2009

\RBibitem{GWWZ}
\by P.~Gwiazda, P.Wittbold, A.Wr\'{o}blewska, A.Zimmermann
\paper Renormalized solutions of nonlinear elliptic problems in generalized Orlicz spaces
\jour J. Differential Equations
\vol 253
\issue
\pages 635–-666
\yr 2012


\RBibitem{Kozh_vmmf_2017}
\by L.M.~Kozhevnikova
\paper
On the entropy solution to an elliptic problem in anisotropic Sobolev--Orlicz spaces
\jour Zh. Vychisl. Mat. Mat. Fiz.
\vol 57
\issue 3
\pages 429–-447
\yr 2017

\RBibitem{Kozh_int_2017}
\by L.M.~Kozhevnikova
\paper Existence of entropic solutions of elliptic problem in anisotropic Sobolev–Orlicz spaces
\jour Itogi Nauki i Tekhniki. Ser. Sovrem. Mat. Pril. Temat. Obz.
\vol 139
\pages 15–-38
\yr 2017

\RBibitem{Muk}
\by F.Kh.~Mukminov
\paper Uniqueness of the renormalized solution of an elliptic-parabolic problem in anisotropic Sobolev--Orlicz spaces
\jour Sb. Math.
\vol 208
\issue 8
\pages 1187--1206
\yr 2017

\RBibitem{Muk_2018}
\by F.Kh.~Mukminov
\paper Existence of a renormalized solution to an anisotropic parabolic problem with variable nonlinearity exponents
\jour Sb. Math.
\vol 209
\issue 5
\pages 714--738
\yr 2018


\RBibitem{DMOP_1999}
\by G.~Dal Maso, F.~Murat, L.~Orsina,  A.~Prignet
\paper Renormalized solutions
of elliptic equations with general measure data
\jour Ann. Scuola Norm. Sup.
Pisa Cl. Sci.
\vol 28
\pages 741–808
\yr 1999


\RBibitem{Bidaut_Veron}
\by M.F.~Bidaut-Veron
\paper Removable singularities and existence for a quasilinear equation with absorption or source term and measure data
\jour Adv. Nonlinear Stud.
\vol 3
\pages 25–-63
\yr 2003

\RBibitem{Malusa}
\by A.~Malusa
\paper A new proof of the stability of renormalized
solutions to elliptic equations with
measure data
\jour Asymptotic Analysis
\vol 43
\pages 111–129
\yr 2005


\RBibitem{Malusa_Porzio}
\by A.~Malusa, M.M.~Porzio
\paper Renormalized solutions to elliptic equations with measure data in
unbounded domains
\jour Nonlinear Analysis
\vol 67
\pages 2370–2389
\yr 2007


\RBibitem{Veron}
\by  L.~Veron
\book  Local and global aspects of quasilinear degenerate elliptic equations
\yr 2017
\publ  World Scientific
\publaddr New Jersey





\RBibitem{Halsey}
\by T.C.~Halsey
\paper Electrorheological fluids
\jour Science
\vol 258
\issue 5083
\pages 761--766
\yr 1992



\RBibitem{Zhik}
\by V.V.~Zhikov
\paper On variational problems and nonlinear elliptic equations with nonstandard growth conditions
\jour J. Math. Sci.
\vol 173
\issue 5
\pages  463--570
\yr 2011

\RBibitem{Alxut}
\by Yu.A.~Alkhutov
\paper The Harnack inequality and the H?older property of solutions of nonlinear elliptic equations with a nonstandard growth condition
\jour Diff. Eq.
\vol   33
\issue{12}
 \pages{1653–1663}
\yr{1997}



\RBibitem{BendWitt}
\by M.~Bendahmane,  P.~Wittboldb
\paper Renormalized solutions for nonlinear elliptic equation with variable exponents and $L^1$-data
\jour Nonlinear Analysis
\vol
\issue
\pages 1--21
\yr 2009


\RBibitem{Urbano}
\by M.~Sancho'n, J.M.~Urbano
\paper Entropy solutions for the $p(x)$-laplace equation
\jour Transactions of the american mathematical society
\vol 361
\issue 12
\pages 6387–-6405
\yr 2009





\RBibitem{Ouaro}
\by S.~Ouaro
\paper Well-Posedness Results for Anisotropic Nonlinear Elliptic
Equations with Variable Exponent and $L^1$-Data
\jour CUBO A Mathematical Journal
\vol 12
\issue 1
\pages 133–-148
\yr 2010

\RBibitem{Zhang_Zhou}
\by C.~Zhang, S.~Zhou
\paper Entropy and renormalized solutions for the p(x)-laplacian equation with mesure data
\jour Bull. Aust. Math. Soc.
\vol 82
\pages 459–479
\yr 2010




\RBibitem{Kozhe_SMFN}
\by L.M.~Kozhevnikova
\paper On entropy solutions of anisotropic elliptic equations with variable nonlinearity indices
\jour Differential and functional differential equations, CMFD
\vol 63
\issue 3
\pages 475–-493
\yr 2017

\RBibitem{Ouaro_2013}
\by I.~Nyanquini, S.~Ouaro,  S.~Soma
\paper Entropy solution to nonlinear multivalued elliptic problem
with variable exponents and measure data
\jour Annals of the University of Craiova, Mathematics and Computer Science Series.
\vol 40
\issue 2
\pages 1–-25
\yr 2013

\RBibitem{Zhang_2014}
\by C.~Zhang
\paper Entropy solutions to nonlinear  elliptic equations
with variable exponents
\jour Electronic Journal of Differential Equations
\vol 2014
\issue 92
\pages 1–-14
\yr 2014







 \RBibitem{AHT}
\by E.~Azroul,  H.~Hjiaj,  A.~Touzani
\paper Existence and regularity of entropy solutions for strongly nonlinear  $p(x)$-elliptic equations
\jour Electronic Journal of Differential Equations
\vol 2013
\issue 68
\pages 1--27
\yr 2013












\RBibitem{Benboubker_2014}
\by M.\,B.~Benboubker,  H.~Hjiaj,  S.~Ouaro
\paper Entropy solutions to nonlinear elliptic anisotropic problem with variable exponent
\jour Journal of Applied Analysis and Computation
\vol 4
\issue 3
\pages 245--270
\yr 2014


\RBibitem{Benboubker_2015}
\by M.B.~Benboubker,   H.~Chrayteh,  M.~El Moumni,   H.~Hjiaj
\paper Entropy and renormalized solutions for nonlinear elliptic problem involving variable exponent and measure data
\jour Acta Mathematica Sinica, English Series
\vol 31
\issue 1
\pages 151--169
\yr 2015



\RBibitem{Azroul_2018}
\by T.~Ahmedatt, E.~Azroul, H.~Hjiaj, A.~Touzani
\paper Existence of Entropy solutions for some nonlinear elliptic problems
involving variable exponent and measure data
\jour Bol. Soc. Paran. Mat.
\vol 36
\issue 2
\pages 33–-55
\yr 2018





\RBibitem{Diening}
\by L.~Diening, P.~Harjulehto, P.~H\"{a}st\"{o}, M.~Ru{z}i{c}ka
\book Lebesgue and Sobolev Spaces with Variable Exponents
\publaddr Berlin, Heidelberg
\publ Springer
\yr 2011

\RBibitem{Edmunds}
\by D. Edmunds,  J. Rakosnik
\paper Sobolev embeddings with variable exponent
\jour Studia Math.
\vol 143
\pages 267–-293
\yr 2000

\RBibitem{Fan}
\by X.~Fan
\paper Anisotropic variable exponent Sobolev spaces and p(x)-Laplacian equations
\jour Complex Variables and Elliptic Equations
\vol 56
\issue 7--9
\pages 623--642
\yr 2011

\RBibitem{Benboub_2011}
\by M.B.~Benboubker, E.~Azroul, A.~Barbara
\paper Quasilinear elliptic problems with nonstandartd  growths
\jour Electronic Journal  of  Differential Equations
\issue{62}
\pages{1--16}
\yr{2011}
	
\RBibitem{KKM}
\by A.\,Sh.~Kamaletdinov, L.\,M.~Kozhevnikova,  L.\,Yu.~Melnik
\paper Existence of solutions of anisotropic elliptic equations
with variable exponents in unbounded domains
\jour Lobachevskii Journal of Mathematics
\vol{39}
\issue{2}
\pages{224–235}
 \yr{2018}


\bibitem{LL} \by J.~Leray, J.\,L.~ Lions
\paper Quelques resultats de
Visik sur les problemes elliptiques non lineaires par les methodes
de Minty-Browder
\jour Bulletin de la S. M. F.
\vol 93
\pages{97--107}
 \yr{1965}

\RBibitem{Danf_Shvarc}
\by  N.~Dunford, J.T.~Schwartz
\book  Linear Operators, Part 1: General Theory
\yr 1958
\publ Interscience Publishers
\publaddr New York

\RBibitem{Lions}
\by J.L.~Lions
\book  Quelques m\'ethodes de r\'esolution des problemes aux limites non lin\'eaires
\yr 1969
\publ Dunod, Gauthier-Villars
\publaddr Paris

\end{thebibliography}
\end{document}